\documentclass[12pt, a4paper]{amsart}
\usepackage{graphicx} 
\usepackage{amsmath,amsthm,amssymb,latexsym,a4wide,tikz,multicol,tikz-cd, calc}
\usepackage{arydshln,multirow,xcolor}
\usepackage{tikz-qtree,tikz-qtree-compat}
\usepackage{mathtools,stmaryrd}
\usepackage[mathscr]{euscript}
\usepackage{comment}
\tikzset{font=\small}
\usepackage{enumerate}
\usetikzlibrary{matrix,arrows,automata,positioning,decorations.pathmorphing}
\usetikzlibrary{cd}
\usetikzlibrary{patterns}
\newtheorem{theorem}{Theorem} [section]
\newtheorem{lemma}[theorem]{Lemma}
\newtheorem{corollary}[theorem]{Corollary}
\newtheorem{proposition}[theorem]{Proposition}

\theoremstyle{definition}
\newtheorem{definition}[theorem]{Definition}
\newtheorem{remark}[theorem]{Remark}

\newcommand{\mx}[1]{#1^{\mathfrak{m}}}

\DeclareMathOperator{\dom}{dom}
\DeclareMathOperator{\ran}{ran}
\newcommand{\Cay}{\operatorname{Cay}}

\title[$F$-birestriction monoids in enriched signature]{$F$-birestriction monoids in enriched signature}


\usepackage[active]{srcltx}

\usepackage{enumerate}
\numberwithin{equation}{section}

\begin{document}
\author{Ganna Kudryavtseva}
\address{G. Kudryavtseva: University of Ljubljana,
Faculty of Mathematics and Physics, Jadranska ulica 19, SI-1000 Ljubljana, Slovenia / Institute of Mathematics, Physics and Mechanics, Jadranska ulica 19, SI-1000 Ljubljana, Slovenia}
\email{ganna.kudryavtseva\symbol{64}fmf.uni-lj.si}
\author{Ajda Lemut Furlani}
\address{A. Lemut Furlani: Institute of Mathematics, Physics and Mechanics, Jadranska ulica 19, SI-1000 Ljubljana, Slovenia/ Faculty of Mathematics and Physics, Jadranska ulica 19, SI-1000 Ljubljana, Slovenia}
\email{ajda.lemut\symbol{64}imfm.si}

\thanks{Ganna Kudryavtseva was supported by the ARIS grants  P1-0288 and J1-60025. Ajda Lemut Furlani was supported by the ARIS grant  P1-0288.}
\sloppy

\begin{abstract}
Motivated by recent interest to $F$-inverse monoids, on the one hand, and to restriction and birestriction monoids, on the other hand, we initiate the study of $F$-birestriction monoids as algebraic structures in the enriched signature $(\cdot, \, ^*, \,^+, \,\mx{},1)$ where the unary operation $\mx{(\cdot)}$ maps each element to the maximum element of its $\sigma$-class.

We find a presentation of the free $F$-birestriction monoid 
${\mathsf{FFBR}}(X)$
as a birestriction monoid ${\mathcal F}$ over the extended set of generators $X\cup\overline{X^+}$ where $\overline{X^+}$ is a set in a bijection with the free semigroup $X^+$ and encodes the maximum elements of (non-projection) $\sigma$-classes. This enables us to show that ${\mathsf{FFBR}}(X)$ decomposes as the partial action product \mbox{$E({\mathcal I})\rtimes X^*$} of the idempotent semilattice of the universal inverse monoid ${\mathcal I}$ of ${\mathcal F}$ partially acted upon by the free monoid $X^*$. Invoking Sch\"utzenberger graphs, we prove that the word problem for ${\mathsf{FFBR}}(X)$ and its strong and perfect analogues is decidable. Furthermore, we show that ${\mathsf{FFBR}}(X)$ does not admit a geometric model based on a quotient of the Margolis-Meakin expansion $M({\mathsf{FG}}(X), X\cup \overline{X^+})$ over the free group ${\mathsf{FG}}(X)$, but the free perfect $X$-generated $F$-birestriction monoid admits such a model.
\end{abstract}
\maketitle
\section{Introduction}
Restriction and birestriction semigroups\footnote{Restriction and birestriction semigroups appeared in the literature as {\em right restriction semigroups} and {\em two-sided restriction semigroups}, respectively. Here, just as in \cite{K24}, we opted to change the terminology and used what is in our opinion more compact and logical terminology, which is in addition consistent with that used in the category theory literature, see, e.g., bisupport categories of \cite{CGH12}.} and monoids  are perhaps the most far-reaching and well studied non-regular generalizations of inverse semigroups. They arise naturally and are useful far beyond semigroup theory, for example, in theoretical computer science and in operator algebras, see \cite{CM24,CG21,CL02,G23b,KL17}. 
$F$-birestriction monoids\footnote{$F$-birestriction monoids were called $F$-restriction in \cite{J16, Kud15}.} form a subclass of birestriction monoids and generalize $F$-inverse monoids, which are especially useful in the theory of $C^*$-algebras \cite{KS02,LN16,Starling16,S11} and, due to recent advances such as the solution of the finite $F$-inverse cover problem \cite{ABO22} and the development of geometric methods \cite{AKSz21, KLF24, Sz24}, have gained major importance within the inverse semigroup theory.

$F$-birestriction monoids were introduced independently by Jones \cite{J16} and the first-named author \cite{Kud15} and are birestriction monoids (for the introduction to the latter, see Section~\ref{sec:prelim}) such that every $\sigma$-class (where $\sigma$ is the minimum congruence which identifies all the projections) has a maximum element with respect to the natural partial order.
It was shown by Kinyon \cite{K18} that $F$-inverse monoids in the enriched signature $(\cdot, \,^{-1}, \mx{},1)$ form a variety of algebras \cite[Corollary 3.2]{AKSz21} and we observe that so do
$F$-birestriction monoids in the enriched signature $(\cdot, \, ^*,\, ^+,\,\mx{},1)$, see Section \ref{sec:varieties}. This raises the question to determine the structure of the free $F$-birestriction monoid ${\mathsf{FFBR}}(X)$ and to solve the word problem for it. 
Since both the free $F$-inverse monoid ${\mathsf{FFI}}(X)$ and the free birestriction monoid ${\mathsf{FBR}}(X)$ admit geometric models based on the Cayley graph of the free group ${\mathsf{FG}}(X)$ (see \cite{AKSz21,FGG09,Kud19,KLF24}), it
is natural to wonder if ${\mathsf{FFBR}}(X)$ admits a similar model. The aim of the present paper is to answer these and some other related questions.

The first-named author showed in~\cite{Kud19} that certain universal proper birestriction monoids (which are analogues of the Birget-Rhodes prefix group expansion \cite{BR89,Sz89}) have sufficient amount of left-right symmetry in that such a monoid $S$ is isomorphic to a partial action product of the idempotent semilattice of the universal inverse monoid of $S$ partially acted upon (by order isomorphisms between order ideals) by the monoid $S/\sigma$. 
 To be able to extend the methods from \cite{Kud19} to the study of $F$-birestriction monoids, we first find, adapting the idea from \cite{KLF24}, the birestriction monoid presentations  of the free $X$-generated $F$-birestriction monoid and its left strong, right strong, strong and perfect analogues.
We prove in particular that ${\mathsf{FFBR}}(X)$ is isomorphic to the birestriction monoid 
$$
{\mathcal F} = {\mathrm{BRestr}} \langle X\cup \overline{X^+} \mathrel{\vert} \overline{x}\geq x, \, x\in X, \, \overline{uv}\geq \overline{u}\,\overline{v},\, u,v\in X^+\rangle
$$
over the extended set of generators $X\cup\overline{X^+}$, where $\overline{X^+}$ is a bijective copy of the free semigroup $X^+$ and encodes the maximum elements of (non-projection) $\sigma$-classes (see Theorem~\ref{th:isom}).
Since $F$-birestriction monoids are proper (\cite[Lemma 5]{Kud15}), the structure result for proper birestriction semigroups (which is recalled in Theorem~\ref{thm:proper}) implies that ${\mathcal F}$ decomposes as a partial action product of its projection semilattice $P({\mathcal F})$, partially acted upon by the monoid $X^*$ (see Theorem~\ref{th:main1}). Building on the methods of \cite{Kud19}, we further show  that $P(\mathcal F)$ is isomorphic to the idempotent semilattice of the universal inverse monoid ${\mathcal I}$ of ${\mathcal F}$. This enables us to invoke the technique developed by Stephen
\cite{S90}, which relies on Sch\"utzenberger graphs, and prove that the word problem for ${\mathcal I}$ and then also for ${\mathsf{FFBR}}(X)$ is decidable (see Theorem \ref{th:word}). We also prove similar results for the free (left, right) strong and perfect $X$-generated $F$-birestriction monoids (see Theorems \ref{th:main1aa} and \ref{thm:word_rel}).

Furthermore, we show that ${\mathsf{FFBR}}(X)$ can not be $(X\cup \overline{X^+})$-canonically embedded, as a birestriction monoid, into an $E$-unitary inverse monoid, which is a quotient of the Margolis-Meakin expansion $M({\mathsf{FG}}(X), X\cup \overline{X^+})$ over the free group ${\mathsf{FG}}(X)$ (see Theorem \ref{th:nogeom} and Proposition \ref{prop:nogeom1}). The free strong $F$-birestriction monoid ${\mathsf{FFBR}}_s(X)$ admits such an embedding if and only if $|X|=1$ (see Theorems \ref{th:nogeom} and  \ref{th:geom1}). Lastly, Theorem \ref{th:geom_p} shows that for the free perfect $F$-birestriction monoid ${\mathsf{FFBR}}_p(X)$ such an embedding exists for any (non-empty) set $X$.

The paper is organized as follows. 
In Section \ref{sec:prelim} we collect the necessary preliminaries. Section \ref{sec:structureFR} provides a self-contained proof of the structure result for the free birestriction monoid ${\mathsf{FBR}}(X)$ \cite{FGG09}, based on the methods of \cite{Kud19}.
This proof is used in Section \ref{sec:projections} to show that  the projection semilattice of a birestriction monoid $S$ given by a presentation is isomorphic to the idempotent semilattice of its universal inverse monoid. In Section \ref{s:geom_models} we study quotients of the Margolis-Meakin expansion $M(G,X)$ of an $X$-generated group $G$, whose maximum group quotient is $G$, via dual-closure operators on the semilattice of connected subgraphs of the Cayley graph of $G$ which contain the origin. We further define and study the notion of a birestriction monoid which has a geometric model based on an $E$-unitary quotient of $M(G,X)$ over $G$.
Further, in Section \ref{sec:varieties} we show that
$F$-birestriction monoids in the enriched signature $(\cdot, \, ^*, \,^+, \mx{}, 1)$ form a variety of algebras and define the varieties of left strong, right strong, strong and perfect  $F$-birestriction monoids.
In Sections~\ref{sec:structureFFR} to~\ref{sec:word_rel} we obtain a coordinatization of ${\mathsf{FFBR}}(X)$ and its  left strong, right strong, strong and perfect analogues, which is then used to show that the word problem for all these $F$-birestriction monoids is decidable.
Finally, Section \ref{sec:geom_models} treats the questions if ${\mathsf{FFBR}}(X)$ and its  left strong, right strong, strong and perfect analogues admit geometric models based on the Cayley graph of the free group ${\mathsf{FG}}(X)$ with respect to generators $X\cup \overline{X^+}$ or $X\cup \overline{X}$.

For the undefined notions in inverse semigroups we refer the reader to \cite{Lawson_book, Petrich}, in birestriction semigroups to \cite{G_notes,G12}, and in universal algebra to \cite{BS81}.

\section{Preliminaries}\label{sec:prelim}
\subsection{Birestriction semigroups}\label{subs:rest_semi}
We start from recalling the necessary background about birestriction semigroups. For more information, see the survey \cite{G_notes}.\footnote{We remind the reader that we have updated the terminology and what we call restriction, corestriction and birestriction semigroups appeared in the literature as right restriction (or weakly right ample) semigroups, left restriction  (or weakly left ample) semigroups and two-sided restriction (or weakly ample) semigroups, respectively.}
\begin{definition} (Restriction and corestriction semigroups)
{\em A restriction semigroup} is an algebra $(S;\, \cdot, \, ^*)$ of type $(2,1)$ such that $(S;\,\cdot)$ is a semigroup and and the following identities hold:
\begin{equation}\label{def:right_rest}
xx^*=x, \, x^* y^*=y^* x^*, \, (xy^*)^*=x^* y^*, \, x^* y=y(xy)^*.
\end{equation}

Dually, a {\em corestriction semigroups} is an algebra $(S;\,\cdot,\, ^+)$ of type $(2,1)$ such that $(S;\,\cdot)$ is a semigroup and 
the following identities hold:
\begin{equation}\label{def:left_rest}
x^+x=x, \, x^+y^+=y^+x^+, \, (x^+y)^+=x^+y^+, \, xy^+=(xy)^+x.
\end{equation}
\end{definition}

\begin{definition} (Birestriction semigroups)
A {\em birestriction semigroup} is an algebra $(S;\,\cdot, \,^*,\,^+)$ of type $(2,1,1)$, where $(S;\cdot, ^+)$ is a restriction semigroup, $(S;\,\cdot, \,^*)$ is a corestriction semigroup and the following identities, which connect the operations $^*$ and $^+$, hold:
\begin{equation}\label{def:rest}
(x^+)^* = x^+, \, (x^*)^+=x^*.
\end{equation}
\end{definition}
If a birestriction semigroup has an identity element, we call it a {\em birestriction monoid} and regard it as an algebra of signature $(\cdot,\, ^*,\, ^+,1)$.
Morphisms and subalgebras of such algebras will be taken with respect to this signature. To emphasize this, we sometimes refer to, e.g., morphisms of birestriction monoids as $(2,1,1,0)$-morphisms. 

Let $X$ be a set, $S$ a birestriction monoid and $\iota_S\colon X\to S$ a map called the {\em assignment map}. We say that $S$ is  {\em $X$-generated} via $\iota_S$ if $S$ is $(2,1,1,0)$-generated  by $\iota_S(X)$ (in the sense of \cite[Definition 3.4]{BS81}). 
 
Let $S$ be a birestriction semigroup. It follows from \eqref{def:rest} that the sets $S^* = \{s^*\colon s\in S\}$ and $S^+ = \{s^+\colon s\in S\}$ are equal. We denote $P(S) = S^* = S^+$ and call elements of $P(S)$ {\em projections}.
It is easy to see that $P(S)$ is a semilattice with $e \wedge f =ef$. Every projection is an idempotent, but an idempotent does not need to be a projection. Note that $e^+=e^*=e$ for all $e\in P(S)$.

Any inverse semigroup $S$ is a birestriction semigroup if one puts $x^* = x^{-1}x$ and $x^+ = xx^{-1}$. In particular, any semilattice $E$ is a birestriction semigroup with $e^*=e^+ =e$ for all $e\in E$.

A {\em reduced birestriction monoid} is a birestriction semigroup with only one projection which is necessarily the identity element. Any monoid $M$ is a reduced birestriction monoid if one puts $x^+=x^*=1$ for all $x\in M$. 

Let $S$ be a birestriction semigroup. One can easily check that the following identities hold for all $s,t\in S$ and $e\in P(S)$:
\begin{equation}\label{eq:identities}
(st)^* =(s^* t)^*, \, (st)^+=(st^+)^+,
\end{equation}
\begin{equation}\label{eq:ample}
es=s(es)^*,  \,  se=(se)^+s,
\end{equation}
\begin{equation}\label{eq:id1}
(se)^*=s^*e, \, (es)^+=es^+.
\end{equation}

The {\em natural partial order} on $S$ is defined by $s \leq t$ if and only if there exists $e \in P(S)$ such that $s=et$ or, equivalently, there exists $f \in P(S)$ such that $s=tf$. It is easy to see (and well known) that $s\leq t$ holds if and only if $s = ts^*$, which is equivalent to $s=s^+t$. Furthermore, if $s,t,u\in S$ and $s \leq t$ then $su \leq tu$, $us \leq ut$, $s^* \leq t^*$ and $s^+ \leq t^+$. We will use these facts throughout the paper without further mention.

The elements $s,t \in S$ are said to be {\em bicompatible}, denoted $s \asymp t$, provided that $st^*=ts^*$ and $t^+s=s^+t$.
The minimum congruence $\sigma$ which identifies all the projections (so that $S/\sigma$ is the maximum reduced quotient of $S$) is given by $s \mathrel{\sigma}t$ provided that $es=et$ for some $e \in P(S)$ or, equivalently, $se=te$ for some $e \in P(S)$. Equivalently, $s\mathrel{\sigma} t$ if and only if there is $u\in S$ such that $u\leq s,t$. We denote the $\sigma$-class of an element $s\in S$ by $[s]_{\sigma}$. Note that $s\asymp t$ implies $s\mathrel{\sigma} t$, for all $s,t\in S$.
 
\begin{definition}\label{def:proper} (Proper birestriction semigroups)
A birestriction semigroup $S$ is called {\em proper}, if the following two conditions hold:
\begin{enumerate}
\item if $s^*=t^*$ and $s \mathrel{\sigma} t$ then $s=t$,
\item if $s^+=t^+$ and $s \mathrel{\sigma} t$ then $s=t$.
\end{enumerate}
\end{definition}

Proper birestriction semigroups generalize $E$-unitary inverse semigroups. It is known and easy to see that $S$ is proper if and only if $\sigma$ coincides with the relation $\asymp$.

The following observation will be needed in the sequel.
\begin{proposition} \label{prop:generation} Let $S$ be an $X$-generated birestriction monoid. Then:
\begin{enumerate}
\item $S$ is $(X\cup P(S))$-generated as a monoid.
\item Every $u\in S$ can be written as $u=ev = wf$ where $e, f\in P(S)$ and $v$, $w$ belong to the $X$-generated submonoid of $S$.
\end{enumerate}
\end{proposition}

\begin{proof}
(1)  We show that every element $s\in S$ can be written as a product of elements from $X\cup P(S)$. If $s=1$, $s\in X$ or $s\in P(S)$, this clearly holds. Otherwise, we can write $s=t_1\cdot t_2$ and the statement easily follows applying inductive arguments.  

(2) Let $u\in S$. By part (1) we can write $u = u_0e_1u_1 \cdots e_nu_n$, where $u_i$ belong to the $X$-generated submonoid of $S$ and $e_i \in P(S)$ for all $i$. In view of \eqref{eq:ample}, the statement follows.
\end{proof}

\subsection{Partial monoid actions by partial bijections and  premorphisms}\label{subs:premor} Here we collect the necessary preliminaries on partial actions of monoids. For a comprehensive survey on partial actions, the reader is referred to \cite{D19}.

The following definition is taken from \cite[Definition 2.2]{Kud19}.
\begin{definition} (Premorphisms)
A {\em premorphism} from a monoid $M$ to a birestriction monoid $S$ is a map $\varphi:M \to S$ such that the following conditions hold:
\begin{itemize}
\item[(PM1)] $\varphi(1)=1$,
\item[(PM2)] $\varphi(m)\varphi(n) \leq \varphi(mn)$, for all $m,n\in M$.
\end{itemize}
\end{definition}
Condition (PM2) is equivalent to each of the following two conditions:
\begin{itemize}
\item[(PM2a)] $\varphi(m)\varphi(n) = \varphi(mn)(\varphi(m)\varphi(n))^*$, for all $m,n\in M$,
\item[(PM2b)] $\varphi(m)\varphi(n) = (\varphi(m)\varphi(n))^+\varphi(mn)$, for all $m,n\in M$.
\end{itemize}

For the next definition, we introduce the following conditions:
\begin{itemize}
\item[(LPM)] $\varphi(m)\varphi(n) = \varphi(m)^+\varphi(mn)$, for all $m,n\in M$,
\item[(RPM)] $\varphi(m)\varphi(n) = \varphi(mn)\varphi(n)^*$, for all $m,n\in M$.
\end{itemize}

\begin{definition} (Left strong, right strong and strong premorphisms)
A premorphism $\varphi \colon M\to S$ from a monoid $M$ to a birestriction monoid $S$ is called
{\em left strong} (resp. {\em right strong}), if it satisfies condition (LPM) (resp. (RPM)). If it is both left and right strong, it is called {\em strong}.
\end{definition}

If $\varphi$ is a premorphism from a monoid $M$ to the symmetric inverse monoid ${\mathcal I}(X)$, we denote $\varphi(m)(x)$ by $\varphi_m(x)$, for $m\in M$ and $x\in X$.
Premorphisms from $M$ to  ${\mathcal I}(X)$ are equivalent to partial actions of $M$ on $X$ by partial bijections, which we now define.

\begin{definition} (Partial actions of monoids)
Let $M$ be a monoid and $X$ a (non-empty) set. We say that $M$ {\em acts partially} on $X$ (from the left) if there exists a partial map (that is, a map defined on a subset of $M\times X$) $M\times X \to X, \, \,(m,x)\mapsto m\cdot x$ which satisfies the following conditions:
\begin{itemize}
\item[(PA1)] if $m\cdot x$ and $n\cdot(m\cdot x)$ are defined, then $nm \cdot x$ is defined and $n\cdot(m\cdot x)=nm \cdot x$, for all $m,n \in M$ and $x\in X$;
\item[(PA2)] $1\cdot x$ is defined and equals $x$, for all $x \in X$.
\end{itemize}
If for all $m\in M$ the partial map $X\to X$, $x\mapsto m\cdot x$ is injective, we say that $M$ acts on $X$ by {\em partial bijections.}
\end{definition}

A partial action $M\times X\to X$ by partial bijections, $(m,x)\mapsto m\cdot x$ corresponds to the premorphism $\varphi\colon M \to {\mathcal I}(X)$ such that $m\cdot x$ is defined if and only if $x\in \dom(\varphi_m)$ in which case $\varphi_m(x) = m\cdot x$. In fact, the premorphism $\varphi$ determines not only the left partial action $\cdot$, but also the right partial action $\circ$ such that $x\circ m$ is defined if and only if $x\in \ran \varphi_m = \dom\varphi_m^{-1}$, in which case $x\circ m = \varphi_m^{-1}(x)$. Each of the partial actions $\cdot$ and $\circ$ determines $\varphi$ and thus each of them determines the other, see \cite{Kud15} for details.

Conditions (LPM) and (RPM) have the following counterparts in terms of partial actions: 
\begin{itemize}
\item[(LSPA)] if $mn \cdot x$ and $(mn \cdot x)\circ m$ are defined, then $m\cdot (n\cdot x)$ is defined, for all $m,n\in M$ and $x\in X$,
\item[(RSPA)] if $mn \cdot x$ and $n\cdot x$ are defined then $m\cdot (n\cdot x)$ is defined, for all $m,n\in M$ and $x\in X$.
\end{itemize}

\begin{definition} (Left strong, right strong and strong partial actions by partial bijections)
We say that the partial action $\cdot$ of a monoid $M$ on a set $X$ by partial bijections is {\em left strong} (resp. {\em right strong} or {\em strong}) if its corresponding premorphism $\varphi$ is left strong (resp. right strong or strong) or, equivalently, if it satisfies (LSPA) (resp. (RSPA) or both (LSPA) and (RSPA)).
\end{definition}

We now recall the well-known definition of a partial action of a group and compare it with the above definitions.

\begin{definition}\label{def:partial_group} (Partial group actions)
Let $G$ be a group and $X$ a set. We say that $G$ 
{\em acts partially} on $X$ if there exists a partial map $\varphi:G\times X \to X, \, \,(g,x)\mapsto g\cdot x$ such that conditions (PA1), (PA2) and the following additional condition hold:
\begin{itemize}
\item[(PA3)] if $m\cdot x$ is defined, then $m^{-1}\cdot (m\cdot x)$ is defined and  $m^{-1}\cdot (m\cdot x)=x$, for all $m\in G$ and $x\in X$.   
\end{itemize}
\end{definition}

\begin{remark} \label{rem:group_pa} Note that a partial action of a group $G$ on a set $X$ is automatically by partial bijections.
It was observed by Megrelishvili and Schr\"oder in \cite{MS04} that if the monoid $M$ is a group,  (PA1), (PA2) and (PA3) hold if and only if (PA1), (PA2) and (LSPA) hold. 
That is, a partial group action is a left strong partial monoid action (of this group). 
One can show that, for a group, (LSPA) is equivalent to (RSPA), so the notions of left strong, right strong and strong partial actions coincide.
Therefore, strong partial monoid actions by partial bijections are a generalization of usual partial group actions. 
\end{remark}

\subsection{The structure of proper birestriction semigroups in terms of partial actions}\label{subs:proper_structure}
Let ${\mathcal P}$ be a poset. A non-empty subset $I$ of $\mathcal{P}$ is said to be an {\em order ideal}, if $x\leq y$ and $y\in I$ imply that  $x \in I$, for all $x,y\in {\mathcal{P}}$. For $p\in {\mathcal{P}}$ we put $p^\downarrow := \{q \in {\mathcal{P}} \colon q \leq p\}$ to be the {\em principal order ideal} generated by $p$. A map $f\colon {\mathcal P}\to {\mathcal Q}$ between posets is called an {\em order isomorphism}, if $x\leq y$ is equivalent to $f(x)\leq f(y)$, for all $x,y\in {\mathcal P}$.

Let now $Y$ be a semilattice and $M$ a monoid. By $\Sigma(Y)$ we denote the inverse semigroup of all order isomorphisms 
between order ideals of $Y$. Partial actions of $M$ on $Y$ by order isomorphisms between order ideals correspond to premorphisms from $M$ to $\Sigma(Y)$. Let $\varphi: M \to \Sigma(Y)$ be such a premorphism.
On the set
$$
Y \rtimes_{\varphi} M = \{(e,m) \in Y \times M \colon e \in \ran\varphi_m\}
$$
define the following operations:
\begin{equation}\label{eq:multiplication}
(e,m)(f,n)=(\varphi_m(\varphi_m^{-1}(e)\wedge f),mn),
\end{equation}
$$
(e,m)^+=(e,1),\,\, (e,m)^*=(\varphi_m^{-1}(e),1).
$$
Then  $Y\rtimes_{\varphi} M$ with the operations $\cdot, \,\,^*$ and $^+$ is a proper birestriction semigroup with  \mbox{$P(Y \rtimes_{\varphi} M) \simeq Y$} via the map $(y,1)\mapsto y$. If $Y$ has the top element $1$ then $Y \rtimes_{\varphi} M$ is a monoid with the identity element $(1,1)$. The congruence $\sigma$ on $Y \rtimes_{\varphi} M$ is given by $(e,m) \mathrel{\sigma} (f,n)$ if and only if $m=n$, so that $(Y \rtimes_{\varphi} M)/\sigma \simeq M$ via the map $(y,m)\mapsto m$. If the partial action $\varphi$ is understood, we will denote $Y \rtimes_{\varphi} M$ simply by $Y \rtimes M$.

Furthermore, each proper birestriction semigroup $S$ arises in this way, as follows. Define the premorphism $\varphi \colon S/\sigma \to \Sigma(P(S))$, called the {\em underlying premorphism} of $S$, by setting, for all $m\in S/\sigma$,
$$\dom\varphi_m = \{ e \in P(S) \colon \text{there exists } s \text{ with } [s]_{\sigma} = m \text{ such that } e \leq s^*\},$$
$\varphi_m(e)=(se)^+$, where $e\in \dom\varphi_m$ and $s$ is such that  $[s]_{\sigma} = m$ and $e \leq s^*$.

\begin{theorem}\cite{CG12, Kud15}\label{thm:proper}
Let $S$ be a proper birestriction semigroup. Then the map $f: S \to P(S) \rtimes_{\varphi} S/\sigma$, given by $s \mapsto (s^+,[s]_\sigma)$, is a $(2,1,1)$-isomorphism. If $S$ is a birestriction monoid then $f$ is a $(2,1,1,0)$-isomorphism.
\end{theorem}
For a generalization of Theorem \ref{thm:proper} to proper extensions of birestriction semigroups, see \cite{DKhK21}.

Let $G$ be a group which acts partially (in the sense of Definition \ref{def:partial_group}), by order isomorphisms between order ideals, on a semilattice $Y$. Then $Y \rtimes G$ has the structure of an inverse semigroup if one defines $(e,g)^{-1} = (\varphi_g^{-1}(e), g^{-1})$. This is an $E$-unitary inverse semigroup with the maximum group quotient isomorphic to $G$. Furthermore, for an $E$-unitary inverse semigroup $S$  the statement of Theorem \ref{thm:proper} specializes to a variation, due to Kellendonk and Lawson \cite{KL04}, of the classical McAlister's $P$-theorem  \cite{M274}.  

\subsection{$F$-birestriction monoids} We recall the following definition \cite{Kud15, J16}.

\begin{definition}  ($F$-birestriction monoids)
A birestriction semigroup is {\em $F$-birestriction}  if every its $\sigma$-class has a maximum element with respect to the natural partial order. 
\end{definition}

It is easy to check (or see \cite{Kud15}) that an $F$-birestriction semigroup is necessarily proper and is a monoid. Let $S$ be an $F$-birestriction monoid and $s\in S$. Following \cite{AKSz21}, we denote the maximum element of the $\sigma$-class of $s$ by $\mx{s}$.\footnote{We alert the reader that the notation $\mx{s}$ differs from the respective notation in each of the works \cite{J16, Kud19, KLF24}.} For an $F$-birestriction monoid $S$ let the map
$\tau : S/\sigma \to S$ be given by  $s \mapsto \mx{t}$ where $t\in S$ is such that $\sigma^{\natural}(t)=s$. By \cite[Lemma 5.1]{Kud19} $\tau$ is a premorphism.

\begin{definition} (Left strong, right strong, strong and perfect $F$-birestriction monoids). We call an $F$-birestriction monoid
\begin{itemize}
\item[-] {\em left strong} (resp. {\em right strong} or {\em strong}) if the premorphism $\tau$ is left strong (resp. right strong or strong),
\item[-] {\em perfect} if the premorphism $\tau$ is a monoid morphism.
\end{itemize}
\end{definition}

It follows from  \cite[Proposition 5.2]{Kud19} that $\tau$ is strong (resp. left strong, right strong or perfect) if and only if the underlying premorphism of $S$ is strong (resp. left strong, right strong or perfect).  Remark \ref{rem:group_pa} implies that any $F$-inverse monoid is automatically strong, so that the classes of left strong, right strong and strong $F$-inverse monoids coincide with the class of all $F$-inverse monoids. Hence, not only the class of $F$-birestriction monoids, but also each of the classes of left strong, right strong and strong $F$-birestriction monoids generalizes the class of $F$-inverse monoids. Note, however, that perfect $F$-inverse monoids form a proper subclass of $F$-inverse monoids (each perfect $F$-inverse monoid is in fact a semidirect product of a semilattice and a group, see \cite{AKSz21}), so that the class of perfect $F$-birestriction monoids generalizes the class of perfect $F$-inverse monoids, but not that of all $F$-inverse monoids. 

\section{The structure of the free birestriction monoid ${\mathsf{FBR}}(X)$}\label{sec:structureFR}
Let $X$ be a non-empty set and $X^{-1} = \{x^{-1}\colon x\in X\}$ a disjoint bijective copy of $X$. By $(X\cup X^{-1})^*$ we denote the {\em free involutive monoid over $X$}. 
The involution $(\cdot)^{-1}$ on $(X\cup X^{-1})^*$ is given by $1^{-1}=1$ and $(x_1\cdots x_n)^{-1} = x_n^{-1}\cdots x_1^{-1}$ where $n\geq 1$ and $x_i\in X\cup X^{-1}$ for all $i$. If $u\in (X\cup X^{-1})^*$ and $S$ is an $X$-generated inverse monoid (in particular a group) by $[u]_S$ we denote the value of $u$ in $S$. We identify elements of the free $X$-generated group ${\mathsf{FG}}(X)$ with reduced words $g\in (X\cup X^{-1})^*$. In this section we denote the value of a reduced word $u\in (X\cup X^{-1})^*$ in the free $X$-generated inverse monoid ${\mathsf{FI}}(X)$ and in ${\mathsf{FG}}(X)$ by the same symbol $u$, it will be always clear from the context elements of which objects are considered. For reduced words $u,v\in (X\cup X^{-1})^*$ we write $(uv)_r$ for the reduced form of $uv$. 

It is well known that ${\mathsf{FI}}(X)$ is an $F$-inverse monoid and the maximum element of the $\sigma$-class which projects down to the reduced word $g\in {\mathsf{FG}}(X)$ equals $g$. Furthermore, ${\mathsf{FI}}(X)/\sigma \simeq {\mathsf{FG}}(X)$ and 
$$
{\mathsf{FI}}(X) \simeq E({\mathsf{FI}}(X)) \rtimes {\mathsf{FG}}(X).
$$
The partial action of ${\mathsf{FG}}(X)$ on $E({\mathsf{FI}}(X))$ 
simplifies to  
\begin{equation}\label{eq:a6aa}
e\in {\mathrm{dom}}\varphi_g \text{ if } e\leq g^* \text{ in which case } \varphi_g(e) = (ge)^+.
\end{equation}
In addition, 
\begin{equation}\label{eq:a6ab}
e\in {\mathrm{dom}}\varphi_g^{-1} \text{ if }e\leq g^+ \text{ in which case } \varphi_g^{-1}(e) = (eg)^*.
\end{equation}
Since
\begin{align*}
(e,u)(f,v) & = (\varphi_u(\varphi_u^{-1}(e)f), (uv)_r) & (\text{by } \eqref{eq:multiplication})\\
& = (\varphi_u((eu)^*f), (uv)_r) & (\text{by }\eqref{eq:a6ab})\\
& = ((u(eu)^*f)^+, (uv)_r) & (\text{by }\eqref{eq:a6aa})\\
& = ((euf)^+, (uv)_r) & (\text{by }\eqref{eq:ample})\\
& = (e(uf)^+, (uv)_r) & (\text{by }\eqref{eq:id1}),
\end{align*}
the operations on $E({\mathsf{FI}}(X)) \rtimes {\mathsf{FG}}(X)$ are given by
\begin{equation}\label{def:oper1}
(e,u)(f,v) =  (e(uf)^+,  (uv)_r), \,\, (e,u)^{-1} = ((eu)^*,u^{-1}),
\end{equation}
\begin{equation}\label{def:oper2}
(e,u)^* = ((eu)^*,1), \,\, (e,u)^+ = (e,1),
\end{equation}
and the identity element of $E({\mathsf{FI}}(X)) \rtimes {\mathsf{FG}}(X)$ is $(1,1)$.
Since the partial action of ${\mathsf{FG}}(X)$ on $E({\mathsf{FI}}(X))$ restricts to $X^*$, the subset 
$$
E({\mathsf{FI}}(X)) \rtimes X^* = \{(e,u)\in E({\mathsf{FI}}(X)) \rtimes {\mathsf{FG}}(X) \colon u\in X^*\}
$$
is closed with respect to the multiplication, the unary operations $^*$ and $^+$ and contains the identity element. It is thus a birestriction submonoid of $E({\mathsf{FI}}(X)) \rtimes {\mathsf{FG}}(X)$.
In this section we show that it is isomorphic to the free birestriction monoid ${\mathsf{FBR}}(X)$. This result was first proved in \cite{FGG09} and then generalized in \cite{Kud19} using  different methods. Here we present a direct and self-contained proof of this result based on the approach of \cite{Kud19}. The constructions involved will play an important role in subsequent sections.
The next lemma is inspired by \cite[Proposition 7.10]{Kud19}.

\begin{lemma}\label{lem:gen_positivecorner}
The birestriction monoid $E({\mathsf{FI}}(X))\rtimes X^*$ is $X$-generated via the map $$x \mapsto (x^+,x).$$
\end{lemma}
\begin{proof}
Let $(e,u)\in E({\mathsf{FI}}(X))\rtimes X^*$ where $u=x_1\cdots x_n\in X^*$. Since $(e,u)=(e,1)(u^+,u)=(e,1)(x_1^+,x_1)\cdots(x_n^+,x_n)$, it suffices to prove that $(e,1)$ can be written as a term in generators. We argue by induction on the smallest length $n$ of a word $v$ over $X\cup X^{-1}$ such that $v^*=v^{-1}v=e$ in ${\mathsf{FI}}(X)$. If $n=0$, $(e,1)=(1,1)$, and we are done. Suppose that $n\geq 1$ and that the claim is proved for $e= v^*$ where $|v|\leq n$. If $e=(wx)^* = (w^*x)^*$, where $|w|=n$ and $x\in X$, we have:
\begin{align*}
(e,1)& =  ((w^*x)^*,1) & (\text{by the choice of } e)\\
& = ((w^*x^+x)^*,1) & (\text{by } \eqref{def:left_rest})\\
& = (((w^*x)^+x)^*,1) & (\text{by } \eqref{eq:id1})\\
& = ((w^*x)^+,x)^* & (\text{by } \eqref{def:oper2})\\
& = ((w^*,1)(x^+,x))^* & (\text{by } \eqref{eq:id1} \text{ and }\eqref{def:oper1}).
\end{align*}
Similarly, if $e=(wx^{-1})^*$, where $|w|=n$ and $x\in X$, we have $(e,1)= ((x^+,x)(w^*,1))^+$. The statement now follows applying the inductive assumption.
\end{proof}

The next result is similar to the well-known result about $E({\mathsf{FI}}(X))\rtimes {\mathsf{FG}}(X)$.

\begin{proposition}\label{prop:a7b}\mbox{}
\begin{enumerate}
\item 
Let $(e,u), (f,v)\in   E({\mathsf{FI}}(X))\rtimes X^*$. Then $(e,u) \mathrel{\sigma} (f,v)$ if and only if $u=v$.  
\item $E({\mathsf{FI}}(X))\rtimes X^*$ is an $F$-birestriction monoid with $(u^+,u)$ being the maximum element in the $\sigma$-class of $(e,u)$.
\end{enumerate}
\end{proposition}
\begin{proof}
(1) It is easy to see that $(e,u)\leq (f,v)$ if and only if $u=v$ and $e\leq f$. So if $(e,u)$ and $(f,v)$ have a common lower bound, we must have $u=v$. On the other hand, if $u=v$ then $(ef,u)\leq (e,u), (f,u)$. 

(2) Since $(e,u)\in E({\mathsf{FI}}(X))\rtimes X^*$ if and only if $e\leq u^+$ and in view of part (1), the element $(u^+,u)$ is the maximum element in its $\sigma$-class.
\end{proof}

\begin{proposition}\label{prop:a7a}
The map $\tau\colon X^* \to   E({\mathsf{FI}}(X))\rtimes X^*$, given by $u\mapsto (u^+,u)$, is a monoid morphism. Consequently, $E({\mathsf{FI}}(X))\rtimes X^*$ is a perfect $F$-birestriction monoid.
\end{proposition}

\begin{proof} For any $u,v\in X^*$ we have
\begin{align*}
(u^+,u)(v^+,v) & =  (u^+(uv^+)^+, uv) & (\text{by } \eqref{def:oper1})\\
& = ((u^+uv^+)^+, uv) & (\text{by }\eqref{eq:id1})\\
& = ((uv)^+, uv) & (\text{by } \eqref{def:left_rest} \text{ and } \eqref{eq:identities}),
\end{align*}
as needed.
\end{proof}

We remark, however, that there is no analogue of Proposition \ref{prop:a7a} for the premorphism $\tau\colon {\mathsf{FG}}(X)\to E({\mathsf{FI}}(X))\rtimes {\mathsf{FG}}(X)$ given by $g\mapsto (g^+,g)$.
Indeed, we have $$\tau(x)\tau(x^{-1})=(x^+,x)(x^*,x^{-1}) = ((xx^{-1})^+, 1) = (x^+, 1) \neq (1,1) = \tau(1) = \tau( (xx^{-1})_r).$$

We now formulate the main theorem of this section.

\begin{theorem} \label{th:structure:f_restr} $E({\mathsf{FI}}(X))\rtimes X^*$ is canonically $(2,1,1,0)$-isomorphic to the free birestriction monoid ${\mathsf{FBR}}(X)$. 
\end{theorem}   

To prove Theorem \ref{th:structure:f_restr}, we prove that $E({\mathsf{FI}}(X))\rtimes X^*$ has the universal property of ${\mathsf{FBR}}(X)$. Let $F$ be an $X$-generated birestriction monoid. We aim to construct a canonical (and necessarily surjective) $(2,1,1,0)$-morphism  $E({\mathsf{FI}}(X))\rtimes X^*\to F$. 

For this, we define the map $D\colon {\mathsf{FI}}(X) \to P(F)$, which is a variation of the construction from \cite{Kud19}.
We first define the map $D\colon (X\cup X^{-1})^*\to P(F)$, $u\mapsto D_u$ recursively on the length $|u|$ of the word $u$. If $|u|=0$ then $u=1$, and we put $D_u=1$. Let now $|u|\geq 1$ and write $u=\Tilde{u}v$ where $|\tilde{u}| =|u|-1$ and $v\in X\cup X^{-1}$. We then put:
\begin{align}\label{def:d}
D_u =
\begin{cases}
(D_{\Tilde{u}}x)^*, &\text{ if } v= x \in X,\\
(xD_{\Tilde{u}})^+, &\text{ if } v=x^{-1} \in X^{-1}.
\end{cases}
\end{align}

\begin{remark}\label{rem:d}
Applying induction and \eqref{eq:identities}, it is easy to see that \eqref{def:d} remains valid if $u=\Tilde{u}v$ where $v\in X^+ \cup (X^+)^{-1}$. \end{remark}

If, for example, $X=\{a,b,c,d,e,f\}$ and $u = ac^{-1}b^{-1}def=a(bc)^{-1}def$ then
$$
D_u = (D_{a(bc)^{-1}}def)^* = ((bc D_a)^+def)^* = ((bc a^*)^+def)^*.
$$

The following is immediate from the definition.

\begin{lemma}\label{lem:7.2} If $D_u=D_v$ then $D_{uw} = D_{vw}$, for any $u,v,w\in (X\cup X^{-1})^*$.
\end{lemma}

We aim to show that the map $D$ factors through the defining relations of ${\mathsf{FI}}(X)$:
$$ uu^{-1}u = u, \,\, uu^{-1}vv^{-1}=vv^{-1}uu^{-1} \,\text{ for all } u,v\in (X\cup X^{-1})^*.
$$

We first prove the following auxiliary lemmas.

\begin{lemma}\label{lem:feb9}
Suppose $u,v  \in (X\cup X^{-1})^*$. Then $D_{uvv^{-1}}=D_uD_{v^{-1}}$.
\end{lemma}

\begin{proof}
We argue by induction on $|v|$. If $|v|=0$ we have $v=1$, and there is nothing to prove. Suppose now that $v=x\tilde{v}$ for $\tilde{v} \in (X\cup X^{-1})^*$ with $|\tilde{v}|\geq 0$ and $x \in X\cup X^{-1}$.  Suppose first that $x\in X$. Then $D_{ux\tilde{v}(x\tilde{v})^{-1}}$ can be rewritten as
\begin{align*}
D_{ux\tilde{v}\tilde{v}^{-1}x^{-1}}&= ( xD_{ux\tilde{v}\tilde{v}^{-1}})^+ & (\text{by \eqref{def:d}})\\
&=(xD_{ux}D_{\tilde{v}^{-1}})^+ & (\text{by the induction hypothesis})\\
&=(x(D_ux)^*D_{\tilde{v}^{-1}})^+ & (\text{by \eqref{def:d}})\\
&=(D_uxD_{\tilde{v}^{-1}})^+ & (\text{by \eqref{eq:ample}})\\
&=D_u(xD_{\tilde{v}^{-1}})^+ & (\text{by \eqref{eq:id1}})\\
&=D_uD_{\tilde{v}^{-1}x^{-1}}=D_uD_{(x\tilde{v})^{-1}} & (\text{by \eqref{def:d}}),
\end{align*}
as needed. If $x \in X^{-1}$, the argument is symmetric. 
\end{proof}

\begin{lemma} Let $s,t, u,v\in (X\cup X^{-1})^*$. Then:
\begin{enumerate}
\item $D_{suu^{-1}vv^{-1}t} = D_{svv^{-1}uu^{-1}t}$,
\item $D_{sut} = D_{suu^{-1}ut}$.
\end{enumerate}
\end{lemma}

\begin{proof}
(1) In view of Lemma \ref{lem:7.2}, it suffices to prove that \begin{equation}\label{eq:5a}
D_{suu^{-1}vv^{-1}} = D_{svv^{-1}uu^{-1}}.
\end{equation}
By Lemma \ref{lem:feb9} the left-hand side of \eqref{eq:5a} is equal to $D_{suu^{-1}}D_{v^{-1}} = D_sD_{u^{-1}}D_{v^{-1}}$ and the right-hand side of \eqref{eq:5a} is equal to $ D_sD_{v^{-1}}D_{u^{-1}}$. These are equal, as projections of $F$ commute.

(2) We argue by induction on $|u|$. If $u=1$, there is nothing to prove. Suppose $|u|\geq 1$ and write $u=\tilde{u}x$, with  $\tilde{u} \in (X\cup X^{-1})^*$ and $x \in X\cup X^{-1}$. We have:
\begin{align*}
D_{s(\tilde{u}x)(\tilde{u}x)^{-1}(\tilde{u}x)t} & =D_{s\tilde{u}xx^{-1}\tilde{u}^{-1}\tilde{u}xt}   & (\text{since } (ab)^{-1} = b^{-1}a^{-1})\\ 
& = D_{s\tilde{u}\tilde{u}^{-1}\tilde{u}xx^{-1}xt} & (\text{by part (1)})\\
& = D_{s\tilde{u}xx^{-1}xt} &  (\text{by the induction hypothesis})\\
& = D_{s\tilde{u}xt} & (\text{by the induction hypothesis}),
\end{align*}
as needed.
\end{proof}

It follows that map $D: (X\cup X^{-1})^* \to P(F)$ gives rise to a well defined map $D\colon {\mathsf{FI}}(X)\to P(F)$. The following is immediate by Lemma \ref{lem:feb9}. 

\begin{proposition}\label{prop:7.5}
For any $v \in {\mathsf{FI}}(X)$ and $e\in E({\mathsf{FI}}(X))$ we have: 
\begin{enumerate}
\item $D_{v^{-1}}=D_{vv^{-1}}=D_{v^+}$,
\item $D_v=D_{v^{-1}v}=D_{v^*}$,
\item $D_vD_e=D_{ve}$.
\end{enumerate} 
\end{proposition}

We record the following important consequence of Proposition \ref{prop:7.5}.

\begin{corollary}\label{cor:hom_d}
The map $D\colon {\mathsf{FI}}(X)\to P(F)$ restricts to the morphism of semilattices $D\colon E({\mathsf{FI}}(X))\to P(F)$.
\end{corollary}

We are now in a position to prove the universal property of $E({\mathsf{FI}}(X)) \rtimes X^*$.

\begin{proposition}\label{prop:isom1}
Let $F$ be an $X$-generated birestriction monoid. The map 
\begin{equation}\label{def:psi}
\Psi \colon E({\mathsf{FI}}(X)) \rtimes X^* \to F, \,\, (e,u) \mapsto D_eu,
\end{equation}
is a canonical $(2,1,1,0)$-morphism. 
\end{proposition}

\begin{proof}
The map $\Psi$ clearly respects the identity element.  We show that it respects the operations $^*,\, ^+$ and the multiplication. Let first $(e,u)\in  E({\mathsf{FI}}(X))\rtimes X^*$. We calculate:
\begin{align*}
(\Psi(e,u))^*&=\left(D_{e}u\right)^* & (\text{by } \eqref{def:psi})\\
&= D_{eu} = D_{(eu)^*} & (\text{by Remark }\ref{rem:d} \text{ and Proposition \ref{prop:7.5}(2)})\\
&=\Psi((eu)^*,1)=\Psi((e,u)^*)  & (\text{by } \eqref{def:psi} \text{ and } \eqref{def:oper2}).
\end{align*}
Since $e\in \ran\varphi_u$, we have $e\leq u^+$. Then:
\begin{align*}
D_{e} & \leq D_{u^+} & (\text{by Corollary  \ref{cor:hom_d}})\\
& = D_{u^{-1}} & (\text{by Proposition \ref{prop:7.5}(1)})\\
& = u^+ & (\text{by } \text{Remark \ref{rem:d}}).
\end{align*}
Therefore, 
$$(\Psi(e, u))^+ = (D_{e}u)^+ = D_{e}u^+ = D_e = \Psi(e, 1) = \Psi((e, u)^+).$$
Let now $(e,u), (f,v)\in  E({\mathsf{FI}}(X))\rtimes X^*$. Applying \eqref{def:oper1}  and \eqref{def:psi}, we have:
\begin{equation}\label{eq:dec20b}
\Psi((e,u)(f,v)) = \Psi(e(uf)^+, uv) = D_{e(uf)^+}uv.
\end{equation}
On the other hand, we have:
\begin{align*}
\Psi(e,u)\Psi(f,v) & = D_{e}uD_{f}v & (\text{by }\eqref{def:psi}) \\
& = D_e(uD_f)^+uv & (\text{by } \eqref{eq:ample})\\
& = D_eD_{fu^{-1}}uv & (\text{by Remark }  \ref{rem:d})\\
& = D_eD_{(uf)^+}uv & (\text{by Proposition \ref{prop:7.5}(1)})\\
& =  D_{e(uf)^+}uv & (\text{by Proposition \ref{prop:7.5}(3)}),
\end{align*}
so that $\Psi$ preserves the multiplication.
Finally, $\Psi$ is canonical since, in view of \eqref{def:psi} and \eqref{def:d}, we have
$\Psi(x^+,x) = D_{x^+}x  =x^+x = x$.
\end{proof}

It follows that $E({\mathsf{FI}}(X))\rtimes X^*$ has the universal property of the free $X$-generated birestriction monoid ${\mathsf{FBR}}(X)$ and is thus canonically isomorphic to ${\mathsf{FBR}}(X)$. This completes the proof of Theorem \ref{th:structure:f_restr}.

We thus have a pair of $X$-canonical isomorphisms
$$
\Psi\colon E({\mathsf{FI}}(X)) \rtimes X^* \to {\mathsf{FBR}}(X), \quad \Psi^{-1}\colon {\mathsf{FBR}}(X) \to E({\mathsf{FI}}(X)) \rtimes X^*.
$$
It follows from \eqref{def:psi} that that $\Psi(e,1) = D_e$, so that $\Psi^{-1}(D_e) = (e,1)$.
Writing
$$
\Psi^{-1}\colon {\mathsf{FBR}}(X) \to E({\mathsf{FI}}(X)) \rtimes X^* \subseteq E({\mathsf{FI}}(X)) \rtimes {\mathsf{FG}}(X) \simeq {\mathsf{FI}}(X)
$$
where the latter isomorphism is canonical, we obtain the canonical embedding 
$$
\psi \colon 
{\mathsf{FBR}}(X) \to {\mathsf{FI}}(X)
$$
satisfying $\psi(D_e) =e$ for all $e\in E({\mathsf{FI}}(X))$.
It follows that
$$
\psi\colon P({\mathsf{FBR}}(X)) \to E({\mathsf{FI}}(X))
$$
is an isomorphism of semilattices whose inverse isomorphism is $$D\colon E({\mathsf{FI}}(X)) \to P({\mathsf{FBR}}(X)).$$

We now summarize the results of this section.

\begin{theorem}\mbox{} \label{th:free_restr}
\begin{enumerate}
\item \label{i:fri1} The birestriction monoid 
$E({\mathsf{FI}}(X)) \rtimes X^*$ is $X$-generated via the map $x\mapsto  (x^+,x)$.
\item \label{i:fri2} $E({\mathsf{FI}}(X)) \rtimes X^*$ (and thus also ${\mathsf{FBR}}(X)$) is perfect $F$-birestriction with $(u^+,u)$ being the maximum element in its $\sigma$-class.
\item \label{i:fri3} The map $X^*\to E({\mathsf{FI}}(X)) \rtimes X^*$, given by $u\mapsto (u^+,u)$, is a monoid morphism.
\item \label{i:fri4} The canonical morphism ${\mathsf{FBR}}(X)\to E({\mathsf{FI}}(X)) \rtimes X^*$ is an isomorphism with the inverse $\Psi\colon E({\mathsf{FI}}(X)) \rtimes X^* \to {\mathsf{FBR}}(X)$ given in \eqref{def:psi}.
\item \label{i:fri5} The canonical $(2,1,1,0)$-morphism 
\begin{equation*}\label{eq:a18a}
\psi\colon {\mathsf{FBR}}(X) \to {\mathsf{FI}}(X)
\end{equation*}
restricts to the isomorphism of semilattices $\psi\colon P({\mathsf{FBR}}(X))\to E({\mathsf{FI}}(X))$ with the inverse isomorphism $D\colon E({\mathsf{FI}}(X))\to P({\mathsf{FBR}}(X))$.
\end{enumerate}
\end{theorem}

\section{Projections of an $X$-generated birestriction monoid}\label{sec:projections}
As usual, we denote presentations of inverse monoids, monoids and groups by ${\mathrm{Inv}}\langle X \mathrel{\vert} R \rangle$, ${\mathrm{Mon}}\langle X \mathrel{\vert} R \rangle$ and ${\mathrm{Gr}}\langle X \mathrel{\vert} R \rangle$. By ${\mathrm{BRestr}}\langle X \mathrel{\vert} R \rangle$ we will similarly denote the birestriction monoid generated by a set $X$ subject to a set of relations $R$ of the form $u=v$ where $u$ and $v$ are elements of ${\mathsf{FBR}}(X)$. By definition ${\mathrm{BRestr}}\langle X \mathrel{\vert} R \rangle$ is the canonical quotient of ${\mathsf{FBR}}(X)$ by the congruence generated by the set $$\{(u,v)\colon u=v \text{ is a relation in } R\}.$$

\begin{lemma}\label{lem:generation2a}
\mbox{}
\begin{enumerate}
\item ${\mathrm{BRestr}} \langle X \mathrel{\vert} R\rangle/\sigma \simeq {\mathrm{Mon}} \langle X \mathrel{\vert} \sigma^{\natural}(R)\rangle$, where $\sigma^{\natural}\colon{\mathsf{FBR}}(X) \to X^*$ is the quotient morphism.
\item ${\mathrm{Inv}} \langle X \mathrel{\vert} R\rangle / \sigma \simeq {\mathrm{Gr}} \langle X \mathrel{\vert} \sigma^{\natural}(R)\rangle$ where $\sigma^{\natural}\colon{\mathsf{FI}}(X) \to {\mathsf{FG}}(X)$ is the quotient morphism.
\end{enumerate}
\end{lemma}

\begin{proof}
Since ${\mathrm{BRestr}} \langle X \mathrel{\vert} R\rangle$ satisfies relations $R$, any its $X$-canonical monoid quotient must satisfy relations $\sigma^{\natural}(R)$, which implies (1).
The second statement follows similarly.
\end{proof}

Let $S$ be a birestriction monoid given by the presentation
$$
S = {\mathrm{BRestr}}\langle X \mathrel{\vert} R \rangle.
$$
The $X$-canonical morphism $\psi\colon {\mathsf{FBR}}(X)\to {\mathsf{FI}}(X)$ maps the relations $R$ to the relations $\psi(R)$ on ${\mathsf{FI}}(X)$.
By $\tilde{S}$ we denote the {\em universal inverse monoid} of $S$ which is given by the presentation 
\begin{equation}\label{eq:univ_inv}
\tilde{S} = {\mathrm{Inv}}\langle X \mathrel{\vert} \psi(R) \rangle.
\end{equation}
By the definition, the identical map on  $X$ extends to an $X$-canonical $(2,1,1,0)$-morphism $\psi\colon S \to \tilde{S}$ and every $X$-canonical $(2,1,1,0)$-morphism from $S$ to an $X$-generated inverse monoid canonically factors  through $\tilde{S}$.
The morphism $\psi\colon S\to \tilde{S}$ restricts to the morphism of semilattices $\psi|_{P(S)}\colon P(S)\to E(\tilde{S})$.

\begin{lemma}
The map $D\colon {\mathsf{FI}}(X)\to P({\mathsf{FBR}}(X))$ induces a well defined map $D\colon \tilde{S}\to P(S)$.
\end{lemma}

\begin{proof} 
Let $u=v$ be a relation from $R$ and $s,t\in {\mathsf{FI}}(X)$. We show that the map $D$ takes $s\psi(u)t$ and $s\psi(v)t$ to the same element of $P(S)$. 

By Proposition \ref{prop:generation} every $u\in {\mathsf{FBR}}(X)$ can be written as $u=ea$ where $e$ is a projection and $a$ is a product (possibly empty) of elements of $X$. So we work with a relation $ea = fb$, where $e,f$ are projections and $a,b\in X^*$, and aim to show that
\begin{equation}\label{eq:aux:18a1}
D_{s\psi(ea)t} = D_{s\psi(fb)t}.
\end{equation}
In view of Lemma \ref{lem:7.2}, we can assume that $t=1$. Since $a,b\in X^*$ and $\psi$ acts identically on $X$, we have that $\psi(X^*)$ is canonically isomorphic to $X^*$.
We have:
\begin{align*}
D_{s\psi(ea)} = D_{s\psi(e)\psi(a)}&= (D_{s\psi(e)} a)^*&(\text{by Remark } \ref{rem:d})\\
&= (D_{s}D_{\psi(e)}a)^*& (\text{by Proposition \ref{prop:7.5}(3)})\\
&= (D_{s}ea)^* &(\text{by Theorem \ref{th:free_restr}\eqref{i:fri5}})
\end{align*}
and similarly $D_{s\psi(fb)} = (D_{s}fb)^*$. Since $ea=fb$ holds in $S$, \eqref{eq:aux:18a1} follows.
\end{proof}

The diagram below illustrates the maps $\psi$ and $D$, along with the canonical quotient maps which are presented by vertical arrows. 
\[\begin{tikzcd}
	{P({\mathsf{FBR}}(X))} && {E({\mathsf{FI}}(X))} \\
	\\
	P(S) && E(\tilde{S})
    \arrow[""', from=1-1, to=3-1]
	\arrow["", from=1-3, to=3-3]
	\arrow["\psi",  bend left=15,from=1-1, to=1-3]
	\arrow["\psi",bend left=15,from=3-1, to=3-3]
	\arrow["D", bend left=15, from=1-3, to=1-1]
    \arrow[ "D",bend left=20, from=3-3, to=3-1]
\end{tikzcd}\]

Applying Theorem \ref{th:free_restr}\eqref{i:fri5}, we obtain the following statement.
\begin{proposition}\label{prop:isom_sem_rel}
The maps $\psi\colon P(S)\to E(\tilde{S})$ and $D\colon E(\tilde{S}) \to P(S)$ are well defined and are mutually inverse  isomorphisms of semilattices.
\end{proposition}

\section{Geometric models of birestriction monoids}\label{s:geom_models}
\subsection{$E$-unitary inverse monoids via dual-closure operators on Cayley graphs of groups} \label{subs:e_unitary} Let $G$ be an $X$-generated group. The edge of the Cayley graph $\Cay(G,X)$ from the vertex $g$ to the vertex $h=g[x]_G$ labeled by $x\in X\cup X^{-1}$ will be denoted by $(g,x,h)$. Let ${\mathcal X}_{X}$ be the set of all finite and connected subgraphs of $\Cay(G,X)$ containing the origin. This is a semilattice with $A \leq B$ if and only if $A\supseteq B$, its top element is the graph $\Gamma_1$ with only one vertex, $1$, and no edges. The {\em Margolis-Meakin expansion} \cite{MM89} $M(G,X)$ of $G$ is the inverse monoid 
$$
M(G,X)  = \{(\Gamma,g)\colon \Gamma \in {\mathcal X}_{X}, g \text{ is a vertex of }\Gamma\}
$$
with the identity element $(\Gamma_1,1)$ and the operations of the multiplication and inversion given by 
$$
(A,g)(B,h) = (A\cup gB, gh), \,\,\, (A,g)^{-1} = (g^{-1}A, g^{-1}).
$$
It is well known (see, e.g., \cite[Proposition 2.1(6)]{KLF24}) that $M(G,X)$ decomposes as a partial action product
\begin{equation}\label{eq:19a1}
M(G,X) = {\mathcal X}_X \rtimes G,
\end{equation}
where $G$ act partially on ${\mathcal X}_X$ so that $g\circ \Gamma$ is defined if and only if $g^{-1}$ is a vertex of $\Gamma$ in which case $g\circ \Gamma$ is the {\em left translation} $g\Gamma$ of $\Gamma$ by $g$.

The underlying premorphism $\varphi\colon G\to \Sigma(E(M(G,X)))$ of $M(G,X)$ is given as follows. For $g\in G$ we have that ${\mathrm{dom}}(\varphi_g)$ consists of all $(\Gamma, 1)$ for which there is $(\Gamma',g)\in M(G,X)$ such that $(\Gamma, 1) \leq (\Gamma',g)^* = (g^{-1}\Gamma',g^{-1})(\Gamma',g) = (g^{-1}\Gamma',1)$, which holds if and only if  $\Gamma'\subseteq g\Gamma$. Since $\Gamma'$ must contain the origin, $\Gamma$ must contain $g^{-1}$ as a vertex. If the latter holds one can put $\Gamma'=g\Gamma$. So ${\mathrm{dom}}(\varphi_g)$ consists of all $(\Gamma,1)\in E(M(G,X))$ where $\Gamma$ has $g^{-1}$ as a vertex. For such a pair $(\Gamma,1)$ we have that $\varphi_g(\Gamma,1) = ((g\Gamma,g)(\Gamma,1))^+ = (g\Gamma,g)^+ = (g\Gamma,1)$. So  $\varphi_g$ just performs the partial action by left translation by $g$ on the first component of $(\Gamma,1)$, described just after \eqref{eq:19a1}. This means that the underlying partial action of $M(G,X)$ is {\em equivalent} to the partial action of $G$ on ${\mathcal X}_X$ by left translations via the isomorphism $E(M(G,X))\to {\mathcal X}_X$, $(\Gamma,1)\mapsto \Gamma$.

Furthermore, $M(G,X)$ is $X$-generated via the assignment map $x\mapsto (\Gamma_x,[x]_G)$ where $\Gamma_x$ is the graph which has two vertices, $1$ and $[x]_G$, and one positive edge $(1,x,[x]_G)$. The universal property of $M(G,X)$ says that for any $X$-generated $E$-unitary inverse monoid $S$ such that $S/\sigma$ is canonically isomorphic to $G$, we have that $S$ is a canonical quotient of $M(G,X)$.
The defining relations of $M(G, X)$ (see \cite[Corollary 2.9]{MM89}) are
\begin{equation}\label{eq:i2}
[u]_{M(G, X)}^2 = [u]_{M(G, X)} \text{ whenever }
[u]_{G} = 1,
\end{equation}
where $u$ runs through $(X\cup X^{-1})^*$. 

For a congruence $\rho$ on  $M(G,X)$ which satisfies the condition 
\begin{equation}\label{eq:e_unitary}
[u]_{M(G,X)} \mathrel{\rho} [v]_{M(G,X)} \Longrightarrow  [u]_{G} = [v]_{G}, \text{ for all } u,v\in (X\cup X^{-1})^*,
\end{equation}
we put $M_{\rho}(G,X) = M(G,X)/\rho$.
Since the canonical quotient morphism $M(G,X)\to G$ factors through $M_{\rho}(G,X)$, we have $M_{\rho}(G,X)/\sigma \simeq G$. 
We obtain the following. 

\begin{proposition}\label{prop:quotient2}
Let $G$ be an $X$-generated group. For an $X$-generated inverse monoid $S$ the following statements are equivalent:
\begin{enumerate}
\item $S$ is a canonical quotient of $M(G,X)$ and $S/\sigma$ is canonically isomorphic to $G$.
\item $S\simeq M_{\rho}(G,X)$ for some congruence $\rho$ on $M(G,X)$ satisfying \eqref{eq:e_unitary}.
\item $S$ is $E$-unitary and $S/\sigma$ is canonically isomorphic to $G$.
\end{enumerate}
\end{proposition}

Let ${\mathcal X}_X^{c}$ be the semilattice of all connected (and not necessarily finite) subgraphs of the Cayley graph $\Cay(G,X)$ which contain the origin with the order given by $A\leq B$ if and only if $A\supseteq B$. A {\em dual-closure operator} $j\colon {\mathcal X}_X^{c} \to {\mathcal X}_X^{c}$ is a function which is:
\begin{itemize}
\item {\em contracting}, that is, $j(\Gamma) \leq \Gamma$, for all $\Gamma \in {\mathcal X}_X^{c}$,
\item {\em monotone}, that is, $\Gamma_1\leq \Gamma_2$ implies $j(\Gamma_1) \leq j(\Gamma_2)$, for all $\Gamma_1, \Gamma_2 \in {\mathcal X}_X^{c}$,
\item {\em idempotent}, that is, $j^2(\Gamma) = j(\Gamma)$, for all $\Gamma \in {\mathcal X}_X^{c}$.
\end{itemize}

We say that a graph $\Gamma \in {\mathcal X}_X^{c}$ is {\em $\rho$-closed} if whenever $[u]_{M(G,X)} \mathrel{\rho} [v]_{M(G,X)}$,
where $u,v\in (X\cup X^{-1})^*$,
we have that for any two vertices $\alpha, \beta$ of $\Gamma$, the graph $\Gamma$ has a path from $\alpha$ to $\beta$ labeled by $u$ if and only if it has a path from $\alpha$ to $\beta$ labeled by $v$.

\begin{lemma}\label{lem:4j1}
Let $\rho$ be a congruence on $M(G,X)$ satisfying \eqref{eq:e_unitary} and $\Gamma\in {\mathcal X}_X^{c}$. Then there is a unique minimal $\rho$-closed graph $\Gamma_{\rho}\in {\mathcal X}_X^{c}$ which contains $\Gamma$. 
\end{lemma}

\begin{proof}
Note that $\Cay(G,X)$ has a path from a vertex $\alpha$ to a vertex $\beta$ labeled by $u\in (X\cup X^{-1})^*$ precisely when $\beta=\alpha[u]_{G}$. If  $[u]_{M(G,X)} \mathrel{\rho} [v]_{M(G,X)}$ then $\alpha[u]_{G} = \alpha[v]_G$, so that $\Cay(G,X)$ is $\rho$-closed and contains $\Gamma$. Furthermore, if $\Gamma_i \in {\mathcal X}_X^{c}$, $i\in I$, is the list of all $\rho$-closed subgraphs of ${\mathcal X}_X^{c}$ which contain $\Gamma$, then the connected component $\Gamma_{\rho}$ of $\cap_{i\in I} \Gamma_i$ which contains $\Gamma$ is a $\rho$-closed connected subgraph, so there is $i$ such that $\Gamma_{\rho}=\Gamma_i$, and the statement follows.
\end{proof}

It is routine to verify that the assignment $\Gamma \mapsto \Gamma_{\rho}$ defines a dual-closure operator on ${\mathcal X}_X^{c}$, which we denote by $\bar{\rho}$.
Moreover, we have that
$$
(\Gamma_1,g_1) \mathrel{\rho} (\Gamma_2,g_2) \iff g_1=g_2 \text{ and } \bar{\rho}(\Gamma_1) = \bar{\rho}(\Gamma_2).
$$
It is easy to see that $\bar{\rho}({\mathcal X}_X)$ is a semilattice and that $G$ acts partially on it by left translations. 
This implies that the map
$M(G,X) \to \bar{\rho}({\mathcal X}_X) \rtimes G$, given by 
$(\Gamma, g) \mapsto (\bar{\rho}(\Gamma), g)$, is a monoid morphism with kernel $\rho$. It follows that 
\begin{equation}\label{eq:closure_on_M}
M_{\rho}(G,X) \simeq \bar{\rho}({\mathcal X}_X) \rtimes G.
\end{equation}
A slightly different, but equivalent, approach to the construction of $M_{\rho}(G,X)$ was recently suggested by Szak\'{a}cs in \cite{Sz24}.

\subsection{Birestriction monoids which admit a geometric model} 
Let $S$ be an $X$-generated birestriction monoid given by the presentation $S = {\mathrm{BRestr}} \langle X \mathrel{\vert} R\rangle$
and $\tilde{S}$ be its universal inverse monoid (see \eqref{eq:univ_inv}). 
It follows from Lemma \ref{lem:generation2a} that the group $\tilde{S}/\sigma$ is the universal group of the $X$-generated monoid $S/\sigma$. That is, the identical map on  $X$ extends to an $X$-canonical monoid morphism $\psi \colon S/\sigma \to \tilde{S}/\sigma$ and every $X$-canonical morphism from $S/\sigma$ to an $X$-generated group canonically factors  through $\tilde{S}/\sigma$. 

The following statements deal with proper birestriction semigroups, see Definition \ref{def:proper}.

\begin{lemma} \label{lem:sigma_classes} 
If $S$ is proper, the morphism $\psi\colon S \to \tilde{S}$ is injective on $\sigma$-classes of $S$.
\end{lemma}

\begin{proof} 
Suppose $a\mathrel{\sigma} b$ and $\psi(a)=\psi(b)$. Then $\psi(a)^+ = \psi(b)^+$ and by Proposition \ref{prop:isom_sem_rel} we have $a^+=b^+$. Since $S$ is proper, we conclude that $a=b$.
\end{proof}

\begin{proposition}\label{prop:injective}\mbox{}
\begin{enumerate}
\item  \label{i:pr1} If the morphism $\psi\colon S\to \tilde{S}$ is injective then so is the morphism $\psi \colon S/\sigma \to \tilde{S}/\sigma$.
\item \label{i:pr2} Suppose $S$ is proper. If the morphism $\psi \colon S/\sigma \to \tilde{S}/\sigma$ is injective then so is the morphism $\psi\colon S\to \tilde{S}$.
\end{enumerate}
\end{proposition}

\begin{proof}
(1) Suppose $\psi\colon S \to \tilde{S}$ is injective and let $a,b\in S$ be such that $\psi(a) \mathrel{\sigma} \psi(b)$. This means that $\psi(a)e=\psi(b)e$ for some $e\in E(\tilde{S})$. Since $e=\psi(D_e)$ by Proposition \ref{prop:isom_sem_rel}, it follows that
$\psi(a)\psi(D_e)=\psi(b)\psi(D_e)$ which is equivalent to $\psi(a D_e)=\psi(bD_e)$. Since $\psi$ is injective, it follows that $a D_e=bD_e$, so that $a \mathrel{\sigma} b$, as needed.

(2) Since $\psi\colon S/\sigma \to \tilde{S}/\sigma$ is injective, it takes elements from distinct $\sigma$-classes to distinct $\sigma$-classes. 
This and Lemma \ref{lem:sigma_classes} imply that $\psi\colon S\to \tilde{S}$ is injective.
\end{proof}

\begin{proposition}\label{prop:e_unitary1}
A birestriction semigroup, which $(2,1,1)$-embeds into an $E$-unitary inverse semigroup, is proper.
\end{proposition}

\begin{proof}
Let $T$ be a birestriction semigroup, $T'$ an $E$-unitary inverse semigroup and $f\colon T\to T'$ a $(2,1,1)$-embedding.
Suppose that $s\mathrel{\sigma} t$ in $T$ and $s^*=t^*$. Then there is $u\in T$ such that $u\leq s,t$. It follows that $f(u) \leq f(s), f(t)$ so that $f(s)\mathrel{\sigma} f(t)$. Since, moreover, $f(s)^* = f(t)^*$, it follows that $f(s) = f(t)$ as $T'$ is $E$-unitary. Then also $s=t$ as $f$ is injective. Similarly, if $s\mathrel{\sigma} t$ in $T$ and $s^+=t^+$, we also get $s=t$. It follows that $T$ is proper. 
\end{proof}

Since $\tilde{S}$ is an $X$-generated inverse monoid with the maximal group image (isomorphic to) the $X$-generated group $\tilde{S}/\sigma$,
the universal property of $M(\tilde{S}/\sigma, X)$ implies that $\tilde{S}$ is $E$-unitary if and only if it is a quotient of $M(\tilde{S}/\sigma, X)$.
If this is the case, \eqref{eq:closure_on_M} implies that
\begin{equation}\label{eq:action22}
\tilde{S} \simeq \bar{\rho}({\mathcal X}_{X}) \rtimes \tilde{S}/\sigma,
\end{equation}
where ${\mathcal X}_{X}$ is the semilattice of all connected subgraphs of $\Cay(\tilde{S}/\sigma, X)$ which contain the origin and $\rho$ is the congruence on $M(\tilde{S}/\sigma, X)$ such that $M(\tilde{S}/\sigma, X)/\rho \simeq \tilde{S}$. Note that
\begin{equation}\label{eq:a22d}
\bar{\rho}({\mathcal X}_{X}) \simeq E(\tilde{S}).
\end{equation}

Suppose that $\tilde{S}$ is $E$-unitary and the morphism $\psi \colon S\to \tilde{S}$  is injective. By Proposition \ref{prop:injective} we have that $\psi \colon S/\sigma \to \tilde{S}/\sigma$ is injective so that $S/\sigma \simeq \psi(S/\sigma)$. Proposition \ref{prop:e_unitary1} implies that $S$ is proper, so that $\psi(S)$ is a proper birestriction submonoid of $\tilde{S}$ which is isomorphic to $S$. The isomorphism $\psi\colon S\to \psi(S)$ restricts to the isomorphism $P(S)\to P(\psi(S))$. On the other hand, Proposition \ref{prop:isom_sem_rel} gives us that $\psi$ restricts to the isomorphism $P(S)\to E(\tilde{S})$. Hence $P(\psi(S)) = E(\tilde{S})$. It follows from Subsection \ref{subs:proper_structure} that the underlying premorphism $\varphi$ of $\psi(S)$ is given as follows. For $t\in \psi(S/\sigma)$ we have that 
\begin{equation}\label{eq:action1a}
{\mathrm{dom}}(\varphi_t) = \{e\in E(\tilde{S})\colon \text{there is } s\in \psi(S) \text{ such that } \sigma^{\natural}(s) = t \text{ and } e\leq s^*\}.
\end{equation}
If $e\in {\mathrm{dom}}(\varphi_t)$ then $\varphi_t(e) = (se)^+$ where $s\in \psi(S)$ is any element such that $\sigma^{\natural}(s) = t$ and $e\leq s^*$. By Theorem \ref{thm:proper} we have that $\psi(S)\simeq E(\tilde{S}) \rtimes_{\varphi} \psi(S/\sigma)$. 
Further, the underlying premorphism $\tilde{\varphi}$ of $\tilde{S}$ is given as follows. For $g\in \tilde{S}/\sigma$ and $e\in E(\tilde{S})$ we have that
\begin{equation}\label{eq:action1b}
{\mathrm{dom}}(\tilde\varphi_g) = \{e\in E(\tilde{S})\colon \text{there is } h\in \tilde{S} \text{ such that } \sigma^{\natural}(h) = g \text{ and } e\leq h^*\}.
\end{equation}
If $e\in {\mathrm{dom}}(\tilde\varphi_t)$ then $\tilde\varphi_t(e) = (he)^+$ where $h\in \tilde{S}$ is any element such that $\sigma^{\natural}(h) = g$ and $e\leq h^*$. 
By Theorem \ref{thm:proper} we have that $\tilde{S} \simeq E(\tilde{S}) \rtimes_{\tilde{\varphi}} \tilde{S}/\sigma$. Comparing the definitions of $\varphi$ and $\tilde{\varphi}$, we see that if $t\in \psi(S/\sigma)$ then \begin{equation}\label{eq:domains}
{\mathrm{dom}}(\varphi_t)\subseteq {\mathrm{dom}}(\tilde{\varphi}_t)
\end{equation} 
and for all $e\in {\mathrm{dom}}(\varphi_t)$ we have that $\varphi_t(e) = \tilde{\varphi_t}(e)$.
Restricting $\tilde{\varphi}$ from $\tilde{S}/\sigma$ to $\psi(S/\sigma)$, we obtain a partial action product 
$E(\tilde{S}) \rtimes_{\tilde{\varphi}} \psi(S/\sigma)$. Our considerations imply the following.
\begin{lemma} \label{lem:varphi_tilde}
$\psi(S) \simeq E(\tilde{S}) \rtimes_{\varphi} \psi(S/\sigma)$ which is a $(2,1,1,0)$-subalgebra of $E(\tilde{S}) \rtimes_{\tilde{\varphi}} \psi(S/\sigma)$. Moreover, $E(\tilde{S}) \rtimes_{\varphi} \psi(S/\sigma) = E(\tilde{S}) \rtimes_{\tilde\varphi} \psi(S/\sigma)$ if and only if for every $t\in \psi(S/\sigma)$ we have that ${\mathrm{dom}}(\varphi_t) = {\mathrm{dom}}(\tilde{\varphi}_t)$.
\end{lemma}

Since the partial actions of $\tilde{S}/\sigma$ on $\bar{\rho}({\mathcal X}_{X})$ and on $E(\tilde{S})$ are equivalent (see Subsection \ref{subs:e_unitary}), Lemma \ref{lem:varphi_tilde} captures the embedding of $S$ into $\tilde{S}$ geometrically. We have motivated the following definition.

\begin{definition} (Geometric model)
Let $S$ be an $X$-generated birestriction monoid. We say that $S$ admits a {\em geometric model} if $S$ canonically $(2,1,1,0)$-embeds into an $X$-generated $E$-unitary inverse monoid.\footnote{While our notion of a geometric model is natural for birestriction monoids, there are instances of other classes of unary and biunary semigroups in the literature which admit geometrically arising models of different kinds (see, e.g., \cite{GG00, G96, HKSz25, Kam11}).}
\end{definition}

Theorem \ref{th:structure:f_restr} implies that ${\mathsf{FBR}}(X)$ admits a geometric model. In view of Proposition \ref{prop:e_unitary1}, if $S$ admits a geometric model, it is necessarily proper.

\begin{proposition} \label{prop:model} 
Let $S$ be an $X$-generated birestriction monoid. 
\begin{enumerate}
\item  $S$ admits a geometric model if and only if $\psi\colon S\to \tilde{S}$ is injective and $\tilde{S}$ is $E$-unitary.
\item If $S$ is proper then it admits a geometric model if and only if the inverse monoid
$\tilde{S}$ is $E$-unitary and the morphism $\psi\colon S/\sigma\to \tilde{S}/\sigma$ is injective.
\end{enumerate}
\end{proposition}

\begin{proof}
(1) Suppose $S$ admits a geometric model and canonically embeds into an $X$-generated $E$-unitary inverse monoid $U$. The universal property of $\tilde{S}$ implies that the embedding of $S$ into $U$ factors through $\tilde{S}$, so that $S$ embeds into $\tilde{S}$ and $\tilde{S}$ embeds into $U$. Since $U$ is $E$-unitary, so is $\tilde{S}$. (Indeed, an inverse monoid $S$ is $E$-unitary if and only if $s\geq e\in E(S)$ implies $s\in E(S)$. Hence, inverse submonoids of $E$-unitary inverse monoids are $E$-unitary.) The reverse implication is immediate.

(2) If $S$ is proper and admits a geometric model, (1) implies that $\tilde{S}$ is $E$-unitary and $\psi\colon S/\sigma\to \tilde{S}/\sigma$ is injective
by Proposition \ref{prop:injective}. 
Conversely, suppose that $\tilde{S}$ is $E$-unitary and the morphism $\psi\colon S/\sigma \to \tilde{S}/\sigma$ is injective. By Proposition \ref{prop:injective}  the morphism $\psi\colon S \to \tilde{S}$ is injective, so that $S$ admits a geometric model.
\end{proof}

\section{Varieties of $F$-birestriction monoids}\label{sec:varieties}

The following statement is inspired by and similar to \cite[Proposition 3.1]{AKSz21}.

\begin{lemma}\label{lem:varietyF}
An algebra $(S;\,\cdot, \,^*,\, ^+,\, \mx{},1)$ is an $F$-birestriction monoid if and only if $(S;\, \cdot,\, ^*,\, ^+,\, 1)$ is a birestriction monoid and the following conditions hold:
\begin{itemize}
\item[(M1)] $\mx{a} \geq a$, for all $a \in S$,
\item[(M2)] $\mx{a}=\mx{(ae)}$, for all $a\in S$ and $e\in P(S)$. 
\end{itemize}
\end{lemma}

\begin{proof}
Let $S$ be an $F$-birestriction monoid. Then (M1) holds by the definition of the operation $\mx{(\cdot )}$.
Since $1\mathrel{\sigma} e$ for all $e\in P(S)$, it follows that $a \mathrel{\sigma} ae$ for all $a\in S$ in $e\in P(S)$, so that (M2) holds as well.

Let now $(S;\,\cdot, \,^*,\, ^+,\, \mx{},1)$ be an algebra for which $(S;\, \cdot, \, ^*,\, ^+,\, 1)$ is a birestriction monoid and axioms (M1) and (M2) hold. 
If $a\mathrel{\sigma} b$, there is $e \in P(S)$ such that $ae=be$. Using (M2), we have $\mx{a} = \mx{(ae)} = \mx{(be)} = \mx{b}$. In view of (M1), this yields $\mx{a}\geq b$, showing that $\mx{a}$ is the maximum element in its $\sigma$-class. This completes the proof.
\end{proof}

It is now easy to deduce the following statement.

\begin{proposition}\label{prop:varietyFR} (Varieties of $F$-birestriction monoids)\mbox{}
\begin{enumerate}
\item $F$-birestriction monoids in the signature $(\cdot,\, ^*,\, ^+, \, \mx{},1)$  form a variety, which we denote by ${\bf FBR}$, of type $(2,1,1,1,0)$. It is defined by all the identities which define the variety of birestriction monoids, along with the identities
\begin{equation}\label{eq:a1n}
\mx{x}x^* =x, \,\, \mx{(xy^*)}=\mx{x}.
\end{equation}
\item Left strong and right strong $F$-birestriction monoids form the subvarieties ${\bf FBR_{ls}}$ and ${\bf FBR_{rs}}$ of the variety ${\bf FBR}$, defined by the identities which define the variety ${\bf FBR}$ along with the identities
\begin{equation}\label{eq:left_s}
\mx{x}\mx{y}=(\mx{x})^+\mx{(xy)} \,\, \text{and}
\end{equation}
\begin{equation}\label{eq:right_s}
\mx{x}\mx{y}=\mx{(xy)}(\mx{y})^*,
\end{equation}
respectively.
\item Strong $F$-birestriction monoids form the subvariety ${\bf FBR_s}={\bf FBR_{ls}} \cap {\bf FBR_{rs}}$ of the variety ${\bf FBR}$. 
\item Perfect $F$-birestriction monoids form the subvariety ${\bf FBR_p}$ of the variety ${\bf FBR}$, defined by the identities which define the variety ${\bf FBR}$ along with the identity
\begin{equation}\label{eq:perf}
\mx{x}\mx{y}=\mx{(xy)}.
\end{equation}
\end{enumerate}
\end{proposition}

From  now on, unless explicitly stated otherwise, we will consider $F$-birestriction monoids as $(2,1,1,1,0)$-algebras. 
Morphisms, congruences and subalgebras of such algebras will be taken with respect to the signature $(\cdot,\, ^*,\, ^+,\, \mx{}, 1)$. To emphasize this, we sometimes refer, e.g., to morphisms of $F$-birestriction monoids as $(2,1,1,1,0)$-morphisms.

We define a {\em reduced $F$-birestriction monoid} as an $F$-birestriction monoid which has only one projection, $1$. It is a reduced birestriction monoid, so that $a^* = a^+ =1$ holds for all its elements and the natural partial order on it is trivial. Hence $\mx{a}=a$ holds for all of its elements. It follows that a reduced $F$-birestriction monoid is a reduced birestriction monoid with the operation $\mx{(\cdot)}$ given by $\mx{a}=a$ for all elements $a$.

Let $S$ be an $F$-birestriction monoid and $\sigma$ be the minimum $(2,1,1,0)$-congruence on $S$ which identifies all the projections. It is easy to see that $a \mathrel{\sigma} b$ if and only if $\mx{a} = \mx{b}$. Thus the $(2,1,1,0)$-quotient map $S\to S/\sigma$ preserves the operation $\mx{(\cdot )}$, so that $\sigma$ is in fact a $(2,1,1,1,0)$-congruence on $S$ and the quotient map $S\to S/\sigma$ is a $(2,1,1,1,0)$-morphism. We will use this fact throughout the paper without further mention.

\section{The structure of the free $F$-birestriction monoid ${\mathsf{FFBR}}(X)$}\label{sec:structureFFR}
\subsection{The coordinatization of ${\mathsf{FFBR}}(X)$}
Let $X$ be a non-empty set and ${\mathsf{FFBR}}(X)$ the free $X$-generated $F$-birestriction monoid. 
Since $F$-birestriction monoids are proper (see \cite[Lemma 5]{Kud15}), Theorem \ref{thm:proper} implies that
$$
{\mathsf{FFBR}}(X) \simeq P({\mathsf{FFBR}}(X)) \rtimes {\mathsf{FFBR}}(X)/\sigma.
$$

We first determine the structure of
${\mathsf{FFBR}}(X)/\sigma$.

\begin{lemma}\label{lem:a20a}
${\mathsf{FFBR}}(X)/\sigma$ is canonically isomorphic to $X^*$.
\end{lemma}

\begin{proof}  
Since $X^*$ is an $X$-generated $F$-birestriction monoid (where $\mx{u}=u$, $u^+=u^*=1$ for all $u\in X^*$), the universal property of ${\mathsf{FFBR}}(X)$ implies that $X^*$ is a canonical $(2,1,1,1,0)$-quotient of ${\mathsf{FFBR}}(X)$, so it is a $(2,1,1,0)$-quotient of ${\mathsf{FFBR}}(X)$ as well.
But any monoid $(2,1,1,0)$-quotient of ${\mathsf{FFBR}}(X)$ is an $X$-generated monoid and so is a quotient of $X^*$. 
\end{proof}

Since the free monoid $X^*$ is cancellative, it follows from \cite[Theorem 2.9]{Kud19} that ${\mathsf{FFBR}}(X)$ is ample (for the definition of an ample birestriction monoid, see, e.g., \cite{Kud19}). Lemma \ref{lem:generation12a} shows that a similar comment applies to the $F$-birestriction monoids considered therein.
 
 \subsection{${\mathsf{FFBR}}(X)$ as a birestriction monoid over extended generators $X\cup \overline{X^+}$} \label{subs:extended}
We start from the following observation.

\begin{proposition} \label{prop:generation1}
Let $F$ be an $X$-generated $F$-birestriction monoid and put $M=F/\sigma$. For every $m\in M$ let $\overline{m}$ be the maximum element of the $\sigma$-class of $F$ which projects onto $m$. We put $\overline{M} = \{\overline{m}\colon m\in M\}$. 
\begin{enumerate}
\item $F$ is $(2,1,1,0)$-generated by $X\cup \overline{M}$.
\item Every element $a\in F$ can be written as $a=e\overline{m}$ for some $e\in P(F)$ and  $m\in M$.
\item Let $T$ be the $X$-generated submonoid of $F$. Then $M$ is a quotient of $T$. 
Consequently, if $M \simeq  X^*$ then $T \simeq X^*$, too.
\end{enumerate}
\end{proposition}

\begin{proof}
(1) This is clear, as for every $a\in F$ there is $m\in M$ such that $\mx{a}=\overline{m}$.

(2) For each $a\in F$ we have $a=a^+\overline{m}$ where $\overline{m}=\mx{a}$.

(3) Let $T$ be the $X$-generated submonoid of $F$. Then $\sigma^{\natural}\colon F\to M$ restricts to a surjective morphism from $T$ onto $M$, so that $M$ is a quotient of $T$. 
\end{proof}

\begin{remark}\label{rem:a20b}
Since $\overline{1} = 1$, the generator $\overline{1}$ in Proposition \ref{prop:generation1}(1) can be omitted, so that $F$ is $(2,1,1,0)$-generated by $X\cup (\overline{M\setminus 1})$.
\end{remark}

\begin{corollary} \label{cor:isom11}
${\mathsf{FFBR}}(X)$ is $(X\cup \overline{X^+})$-generated as a birestriction monoid via the assignment map  $x\mapsto x$, where $x\in X$, and $\overline{u} \mapsto \mx{u}$, where $u\in X^+$.
\end{corollary}

\begin{proof} By Lemma \ref{lem:a20a} we have that ${\mathsf{FFBR}}(X)/\sigma \simeq X^*$, so that the statement follows applying Remark \ref{rem:a20b}.
\end{proof}

We define the set of relations
\begin{equation}\label{def:N}
N = \{\overline{x} \geq x, \, \overline{uv} \geq \overline{u}\,\overline{v}\colon x\in X, \,  u,v\in X^+\},
\end{equation}
where $\overline{x}\geq x$ is the abbreviation for the relation $x = \overline{x}x^*$, which is equivalent to the relation $x=x^+\overline{x}$. Similarly, $\overline{uv}\geq \overline{u}\,\overline{v}$ is the abbreviation for the relation $\overline{u}\,\overline{v} = \overline{uv}(\overline{u}\,\overline{v})^*$ or its equivalent relation $\overline{u}\,\overline{v} = (\overline{u}\,\overline{v})^+\overline{uv}$. We put
$$
{\mathcal F} = {\mathrm{BRestr}} \langle X\cup \overline{X^+} \mathrel{\vert} N\rangle.
$$

Applying induction, it is easy to see that $\overline{u_1\cdots u_n} \geq \overline{u_1}\cdots \overline{u_n}$ holds in ${\mathcal F}$ for all $n\geq 2$ and $u_1,\dots, u_n \in X^+$.
Observe that for any $u=x_1\cdots x_n\in X^+$ (where $x_i\in X$ for all $i$) we have that $\overline{u} = \overline{x_1\cdots x_n} \geq \overline{x_1}\cdots \overline{x_n} \geq x_1\cdots x_n = u$. We have proved the following.

\begin{lemma}\label{lem:u13d}
The inequality $\overline{u} \geq u$ holds in ${\mathcal F}$ for any $u\in X^+$. 
\end{lemma}

\begin{lemma}\label{lem:quotient1}
${\mathsf{FFBR}}(X)$ is an $(X\cup \overline{X^+})$-canonical $(2,1,1,0)$-quotient of ${\mathcal F}$.    
\end{lemma}

\begin{proof} Since the defining relations of ${\mathcal F}$ hold in ${\mathsf{FFBR}}(X)$ (with respect to the assignment map of Corollary \ref{cor:isom11}), the claim follows.
\end{proof}

We now aim to show that ${\mathsf{FFBR}}(X)$ is in fact isomorphic to ${\mathcal F}$.
\begin{lemma}\label{lem:20a3}
${\mathcal F}$ is an $F$-birestriction monoid with ${\mathcal F}/\sigma \simeq X^*$. The maximum element of the $\sigma$-class which projects onto $w\in X^*$ equals $\overline{w}$, if $w\in X^+$, and $1$, if $w=1$. Furthermore, ${\mathcal F}$ is $X$-generated as an $F$-birestriction monoid.    
\end{lemma}

\begin{proof}
Note that $X^*$ is an $(X\cup \overline{X^+})$-generated monoid under the assignment map $x\mapsto x$, $x\in X$ and $\overline{w}\mapsto w$, $w\in X^+$, and is an $(X\cup \overline{X^+})$-generated reduced birestriction monoid which satisfies relations $N$. It follows that $X^*$ is an $(X\cup \overline{X^+})$-canonical quotient of ${\mathcal F}$. Note that in any reduced quotient of ${\mathcal F}$ the relations $N$ hold, in particular $\overline{w}=w$ holds for all $w\in X^+$. It follows that any monoid quotient of ${\mathcal F}$ is in fact $X$-generated, so that it is a quotient of $X^*$. We have shown that ${\mathcal F}/\sigma\simeq X^*$.

Let $u \in {\mathcal F}$. Applying Proposition \ref{prop:generation} we write $u=ev$ where $e$ is a projection and $v\in (X\cup \overline{X^+})^*$. If $v=1$, we have $\mx{u}=\mx{e}=1$. Otherwise, write $v = a_0\overline{b_1}a_1\overline{b_2}\cdots a_{n-1}\overline{b_n}a_{n}$ where $n\geq 0$, $b_i\in X^+$ and $a_i\in X^*$ for all $i$. 
Applying the defining relations of ${\mathcal F}$ and Lemma \ref{lem:u13d}, we have
\begin{equation}\label{eq:apr3a}
u \leq v \leq \,\, \overline{a_0}\,\overline{b_1} \cdots \overline{b_n}\,\overline{a_{n}}\leq \overline{a_0b_1\cdots b_na_{n}} = \overline{\sigma^{\natural}(u)},
\end{equation} 
which shows that the $\sigma$-class of $u$ has a maximal element, namely, the element $\overline{\sigma^{\natural}(u)}$. 

Finally, ${\mathcal F}$ is $X$-generated as an $F$-birestriction monoid  since it is $(X\cup \overline{X^+})$-generated as a birestriction monoid and $\overline{w} = \mx{w}$ for all $w\in X^+$.
\end{proof}

\begin{theorem}\label{th:isom} 
${\mathsf{FFBR}}(X)$ and ${\mathcal F}$ are  $(X\cup \overline{X^+})$-canonically $(2,1,1,0)$-isomorphic and $X$-canonically $(2,1,1,1,0)$-isomorphic.
\end{theorem}

\begin{proof}
By the universal property of ${\mathsf{FFBR}}(X)$ and by Lemma \ref{lem:20a3} ${\mathcal F}$ is an $X$-canonical $(2,1,1,1,0)$-quotient and an $(X\cup \overline{X^+})$-canonical $(2,1,1,0)$-quotient of ${\mathsf{FFBR}}(X)$. Combining this
with Lemma \ref{lem:quotient1}, the first claim follows. Since the canonical quotient map of Lemma \ref{lem:quotient1} respects the maximum elements of $\sigma$-classes, the second claim follows, too.
\end{proof}

\subsection{The inverse monoid ${\mathcal I}$ and the projection semilattice of ${\mathsf{FFBR}}(X)$}
\label{subs:i}
From now on let $\psi$ denote the $(2,1,1,0)$-morphism 
$$
\psi\colon {\mathsf{FBR}} (X\cup \overline{X^+})\to {\mathsf{FI}}(X\cup \overline{X^+}),
$$
which is identical on $X\cup \overline{X^+}$.
We put
\begin{equation}\label{eq:a22e}
{\mathcal I} = {\mathrm{Inv}}\langle X\cup \overline{X^+} \mathrel{\vert} \psi(N) \rangle
\end{equation}
to be the  {\em universal inverse monoid}  of the birestriction monoid ${\mathcal F}$. That is, the identical map on  $X\cup \overline{X^+}$ extends to a $(X\cup \overline{X^+})$-canonical morphism of birestriction monoids ${\mathcal F} \to {\mathcal I}$ and every such a morphism from ${\mathcal F}$ to an $(X\cup \overline{X^+})$-generated inverse monoid factors  canonically through ${\mathcal I}$.

It follows from Proposition \ref{prop:isom_sem_rel} that
the isomorphisms 
$$
D: E({\mathsf{FI}}(X\cup \overline{X^+})) \to P({\mathsf{FBR} (X\cup \overline{X^+})})
$$ and
$$
\psi\colon P({\mathsf{FBR}}(X\cup \overline{X^+})) \to E({\mathsf{FI}}(X\cup \overline{X^+}))
$$
induce isomorphisms $D\colon E({\mathcal I}) \to P({\mathcal F})$ and $\psi\colon P({\mathcal F})\to E({\mathcal I})$.
Bearing in mind Theorem \ref{th:isom}, we have 
$P({\mathsf{FFBR}}(X))\simeq E({\mathcal I})$. Combining this with Theorem \ref{thm:proper}, Lemma \ref{lem:a20a} and Lemma \ref{lem:20a3}, we obtain the following result.
\begin{theorem} \label{th:main1}
Let $X$ be a non-empty set. The map $x\mapsto (x^+, x)$ extends to an  $X$-canonical $(2,1,1,1,0)$-isomorphism
$$
{\mathsf{FFBR}}(X) \simeq E({\mathcal I})\rtimes X^*.
$$

In detail, we have 
$$E({\mathcal I})\rtimes X^* = \{(e,u) \in E({\mathcal I})\times X^* \colon e\leq \overline{u}^+\}
$$
with the identity element $(1,1)$ and the remaining operations  given by
$$
(e,u)(f,v)=(e(\overline{u}f)^+,uv),
$$
$$
(e,u)^*=((e\overline{u})^*,1),\,\,(e,u)^+=(e,1) ,
$$
$$
\mx{(e,u)}=(\overline{u}^+,u).
$$
\end{theorem}

\section{The word problem for ${\mathsf{FFBR}}(X)$ is decidable}\label{sec:wordproblem}
Theorem \ref{th:main1} reduces the word problem for ${\mathsf{FFBR}}(X)$ to that for the inverse monoid ${\mathcal I}$, so that we can apply the techniques developed by Stephen \cite{S90} to tackle it. For the undefined terminology used below we refer the reader to \cite{S90}.

\begin{theorem}\label{th:word}
Sch\"utzenberger graphs of elements of ${\mathcal I}$ are finite and effectively constructed. Consequently, the word problem for ${\mathcal I}$ and for ${\mathsf{FFBR}}(X)$ is decidable.
\end{theorem}

\begin{proof}
Let $u\in (Y\cup Y^{-1})^*$ where $Y=X\cup \overline{X^+}$. We aim to construct a finite series of elementary $P$-expansions and determinizations of the linear graph of $u$ which terminates in a finite closed graph. This graph is then ($V$-isomorphic) to the Sch\"utzenberger graph of $[u]_{{\mathcal I}}$ \cite[Theorem 5.10]{S90}.

We start from the linear graph of $u$ and apply several determinizations to obtain its determinized form which is necessarily a tree. If it is closed, we are done. Otherwise, we apply to it a series of elementary $P$-expansions and determinizations, as follows. We consider two possible cases of an elementary $P$-expansion which can be applied.

{\em Case 1.} Suppose that the elementary $P$-expansion stems from the relation $\overline{x}\geq x$, that is, $x=xx^{-1}\overline{x}$, where $x\in X$. The only possibility when this causes attaching a new path is when there is some edge $(\alpha,x,\beta)$ in the graph, but there is no edge $(\alpha,\overline{x}, \beta)$. Then after performing an elementary $P$-expansion (attaching the path from $\alpha$ to $\beta$ labeled by $xx^{-1}\overline{x}$) and several determinizations, we come to the graph obtained from the initial graph by adding the edge $(\alpha,\overline{x}, \beta)$, see Figure  \ref{fig:c1a}\footnote{On pictures, we draw the edges labeled by $X$ as solid lines and those labeled by ${\overline{X^+}}$ as dashed lines.}. If this graph is not determinized, which is the case when the initial graph contains an edge of the form $(\alpha, \overline{x}, \gamma)$ or $(\gamma, \overline{x},\beta)$, we determinize it and obtain a determinized graph with at most the same number of vertices as the initial graph and a fewer number of  edges of the form $(\alpha,x, \beta)$ such that the edge $(\alpha,\overline{x}, \beta)$ is not in the graph. 

\begin{figure}[!ht]
\centering
\begin{tikzpicture}[scale=0.6]
\begin{scope}[every node/.style={circle,fill,inner sep=1.5pt}]
\node[label= {[label distance=0.1cm]:$\alpha$}] (A) at (1,0) {};
\node[label= {[label distance=0.1cm]:{$\beta$}}] (B) at (6.5,0) {}; 
\end{scope}

\begin{scope}[>={Stealth[black]}, every node/.style={fill=white,circle,scale=0.7}, every edge/.style={draw, thick}]
\path [->] (A) edge[bend left=25] node[above=0.5pt] {$x$} (B);
\end{scope}
 
\node (P) at (8, 0) {};
\node (Q) at (10.8, 0) {};
\begin{scope},
\draw[-stealth,thick,decorate,decoration={snake,amplitude=.3mm}] (P) -- (Q) node[above,midway]{};
\end{scope}

\begin{scope}[every node/.style={circle,fill,inner sep=1.5pt},xshift=12cm]
\node[label={[label distance=0.1cm]:$\alpha$}] (A) at (0,0) {}; 
\node[label={[label distance=0.1cm]:}] (B) at (2.5,-0.5) {};
\node[label={[label distance=0.1cm]:}] (C) at (5,-0.5) {};
\node[label={[label distance=0.1cm]:$\beta$}] (D) at (7.5,0) {};
\end{scope}

\begin{scope}[>={Stealth[black]}, every node/.style={fill=white,circle,scale=0.7}, every edge/.style={draw, thick}]
\path [->] (A) edge[bend left=25] node[above=0.8pt] {$x$} (D);
\path [->] (A) edge[bend right=25] node[below=0.8pt]{$x$} (B);
\path [<-] (B) edge[bend right=25] node[below=0.8pt]{$x$} (C);
\path [->] (C) edge[dashed,bend right=25] node[below=0.8pt]{$\overline{x}$} (D);
\end{scope}

\node (P) at (8, -4) {};
\node (Q) at (10.8, -4) {};
\begin{scope},
\draw[-stealth,thick,decorate,decoration={snake,amplitude=.3mm}] (P) -- (Q) node[above,midway]{};
\end{scope}

\begin{scope}[every node/.style={circle,fill,inner sep=1.5pt},xshift=13.2cm]
\node[label={[label distance=0.1cm]:$\alpha$}] (A) at (0,-4) {}; 
\node[label={[label distance=0.1cm]:{$\beta$}}] (B) at (5.5,-4) {}; 
\end{scope}

\begin{scope}[>={Stealth[black]}, every node/.style={fill=white,circle,scale=0.7}, every edge/.style={draw, thick}]
\path [->] (A) edge[bend left=25] node[above=0.8pt] {$x$} (B);
\path [->] (A) edge[dashed,bend right=25] node[below=0.8pt]{$\overline{x}$} (B);
\end{scope}
\end{tikzpicture}
\caption{Illustration of Case 1.}\label{fig:c1a}
\end{figure}
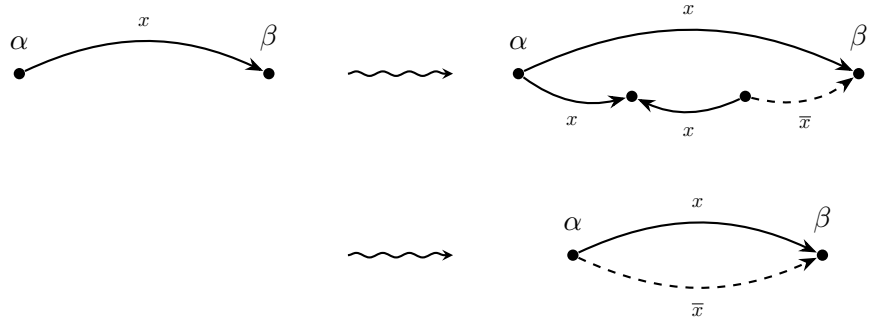
{\em Case 2.} Suppose that the elementary $P$-expansion stems from the relation $\overline{uv} \geq \overline{u}\, \overline{v}$ where $u,v\in X^+$, that is,   $\overline{u}\,\overline{v} = \overline{u}\,\overline{v}(\overline{u}\,\overline{v})^{-1}\overline{uv}$.  The only possibility when this causes attaching a path is when there are edges $(\alpha,\overline{u},\beta)$ and $(\beta,\overline{v},\gamma)$ in the graph, but there is no edge $(\alpha,\overline{uv}, \gamma)$. After performing an elementary $P$-expansion (attaching a path from $\alpha$ to $\beta$ labeled by  $\overline{u}\,\overline{v}(\overline{u}\,\overline{v})^{-1}\overline{uv}$) and several determinizations, we come to the graph obtained from the initial graph by adding the edge $(\alpha,\overline{uv}, \gamma)$, see Figure \ref{fig:c2a}. Similarly as in the previous case, we determinize the obtained graph (if applicable) and obtain a graph with at most the same number of vertices as the initial graph and with fewer number of pairs of edges of the form $(\alpha,\overline{u},\beta)$ and $(\beta,\overline{v},\gamma)$ such that the edge $(\alpha,\overline{uv}, \gamma)$ is not in the graph.

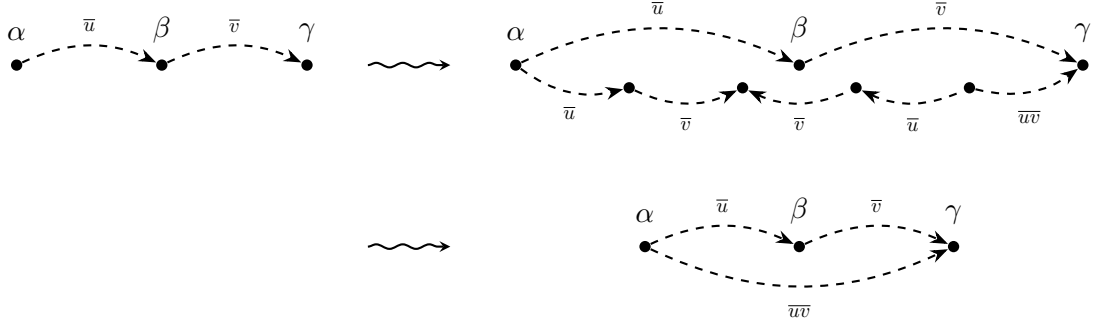
\begin{figure}[!ht]
\begin{tikzpicture}[scale=0.6]
\begin{scope}[every node/.style={circle,fill,inner sep=1.5pt}]
\node[label= {[label distance=0.1cm]:$\alpha$}] (A) at (0,0) {};
\node[label= {[label distance=0.1cm]:{$\beta$}}] (B) at (3.2,0) {};
\node[label= {[label distance=0.1cm]:{$\gamma$}}] (C) at (6.4,0) {};
\end{scope}

\begin{scope}[>={Stealth[black]}, every node/.style={fill=white,circle,scale=0.7}, every edge/.style={draw, thick}]
\path [->] (A) edge[dashed,bend left=25] node[above=0.5pt] {$\overline{u}$} (B);
\path [->] (B) edge[dashed, bend left=25] node[above=0.5pt] {$\overline{v}$} (C);
\end{scope}

\node (P) at (7.5, 0) {};
\node (Q) at (9.8, 0) {};
\begin{scope},
\draw[-stealth,thick,decorate,decoration={snake,amplitude=.3mm}] (P) -- (Q) node[above,midway]{};
\end{scope}

\begin{scope}[every node/.style={circle,fill,inner sep=1.5pt},xshift=11cm]
\node[label={[label distance=0.1cm]:{$\alpha$}}] (A) at (0,0) {}; 
\node[label={[label distance=0.1cm]:{$\beta$}}] (B) at (6.25,0) {};
\node[label={[label distance=0.1cm]:{$\gamma$}}] (C) at (12.5,0) {};
\node[label={[label distance=0.03cm]:}] (D) at (2.5,-0.5) {};
\node[label={[label distance=0.03cm]:}] (E) at (5,-0.5) {};
\node[label={[label distance=0.03cm]:}] (F) at (7.5,-0.5) {};
\node[label={[label distance=0.03cm]:}] (G) at (10,-0.5) {};
\end{scope}

\begin{scope}[>={Stealth[black]}, every node/.style={fill=white,circle,scale=0.7}, every edge/.style={draw, thick}]
\path [->] (A) edge[dashed,bend left=25] node[above=0.5pt] {$\overline{u}$} (B);
\path [->] (B) edge[dashed, bend left=25] node[above=0.5pt] {$\overline{v}$} (C);
\path [->] (A) edge[dashed, bend right=25] node[below=0.5pt] {$\overline{u}$} (D);
\path [->] (D) edge[dashed, bend right=25] node[below=0.5pt] {$\overline{v}$} (E);
\path [<-] (E) edge[dashed, bend right=25] node[below=0.5pt] {$\overline{v}$} (F);
\path [<-] (F) edge[dashed, bend right=25] node[below=0.5pt] {$\overline{u}$} (G);
\path [->] (G) edge[dashed, bend right=25] node[below=0.5pt] {$\overline{uv}$} (C);
\end{scope}
 
\node (P) at (7.5, -4) {};
\node (Q) at (9.8, -4) {};
\begin{scope},
\draw[-stealth,thick,decorate,decoration={snake,amplitude=.3mm}] (P) -- (Q) node[above,midway]{};
\end{scope}
\begin{scope}[every node/.style={circle,fill,inner sep=1.5pt},xshift=14.75cm]
\node[label={[label distance=0.1cm]:{$\alpha$}}] (A) at (-0.9,-4) {}; 
\node[label={[label distance=0.1cm]:{$\beta$}}] (B) at (2.5,-4) {};
\node[label={[label distance=0.1cm]:{$\gamma$}}] (C) at (5.9,-4) {};
\end{scope}

\begin{scope}[>={Stealth[black]}, every node/.style={fill=white,circle,scale=0.7}, every edge/.style={draw, thick}]
\path [->] (A) edge[dashed,bend left=25] node[above=0.5pt] {$\overline{u}$} (B);
\path [->] (B) edge[dashed, bend left=25] node[above=0.5pt] {$\overline{v}$} (C);
\path [->] (A) edge[dashed, bend right=25] node[below=0.5pt] {$\overline{uv}$} (C);
\end{scope}

\end{tikzpicture}
\caption{Illustration of Case 2.}\label{fig:c2a}
\end{figure}

It follows that in a finite number of steps we obtain a closed graph 
with at most the same number of vertices as the initial graph.  
This completes the proof.
\end{proof}

\section{The structure of the free (left, right) strong and perfect 
$F$-birestriction monoids}\label{s:arbitrary_structure}
Denote by ${\mathsf{FFBR}_{ls}}(X)$, ${\mathsf{FFBR}_{rs}}(X)$, ${\mathsf{FFBR}_s}(X)$ and ${\mathsf{FFBR}_{p}}(X)$ the free objects of the varieties  ${\mathbf{FBR}_{ls}}$, ${\mathbf{FBR}_{rs}}$, ${\mathbf{FBR}_s}$, ${\mathbf{FBR}_p}$, respectively. In this section we prove an analogue of Theorem \ref{th:main1} for these objects.

\begin{lemma}\label{lem:generation12a} ${\mathsf{FFBR}}_i(X)/\sigma \simeq X^*$,  for each $i\in \{ls, rs, s,p\}$. 
\end{lemma}

\begin{proof}
The proof is similar to the proof of Lemma \ref{lem:a20a}.  \end{proof}

By ${\mathrm{FBRestr}}\langle X \mathrel{\vert} R \rangle$ we denote the $F$-birestriction monoid generated by a set $X$ subject to the set of relations $R$ of the form $u=v$ where $u,v\in {\mathsf{FFBR}}(X)$.

\begin{proposition}\label{prop:presenations_f}\mbox{}
\begin{enumerate}
\item ${\mathsf{FFBR}_{ls}}(X) = {\mathrm{FBRestr}}\langle X \mathrel{\vert} \mx{u}\mx{v} = (\mx{u})^+\mx{(uv)}, \, u,v \in X^+\rangle$.
\item ${\mathsf{FFBR}_{rs}}(X) = {\mathrm{FBRestr}}\langle X \mathrel{\vert} \mx{u}\mx{v} = \mx{(uv)}(\mx{v})^*, \, u,v \in X^+\rangle$.
\item ${\mathsf{FFBR}_{s}}(X) = {\mathrm{FBRestr}}\langle X \mathrel{\vert} \mx{u}\mx{v} = (\mx{u})^+\mx{(uv)} = \mx{(uv)}(\mx{v})^*,\, u,v \in X^+\rangle$.
\item ${\mathsf{FFBR}_{p}}(X) = {\mathrm{FBRestr}}\langle X \mathrel{\vert} \mx{u}\mx{v} = \mx{(uv)}, \, u,v \in X^+\rangle$.
\end{enumerate}
\end{proposition}

\begin{proof}
(1) Put $F={\mathrm{FBRestr}}\langle X \mathrel{\vert} \mx{u}\mx{v} = (\mx{u})^+\mx{(uv)}, \, u,v \in X^+\rangle$. It suffices to show that the identity \eqref{eq:left_s} is satisfied in $F$. Since the defining relations of $F$ hold in $X^*$, the identity map on $X$ extends to the $(2,1,1,1,0)$-morphism $F\to X^*$. Since any monoid quotient of $F$ is an $X$-generated monoid, we have that $F/\sigma \simeq X^*$.
By Proposition \ref{prop:generation1}(3) the $X$-generated submonoid of $F$ coincides with $X^*$. Let $u,v\in F$ and put $u' = \sigma^{\natural}(u) \in X^*$ and $v'=\sigma^{\natural}(v)\in X^*$. Since $u \mathrel{\sigma} u'$, $v \mathrel{\sigma} v'$ and $uv \mathrel{\sigma} u'v'$, it follows that $\mx{u} = \mx{(u')}$, $\mx{v} = \mx{(v')}$ and $\mx{(uv)} = \mx{(u'v')}$. Then $\mx{u}\mx{v} =\mx{(u')}\mx{(v')} = (\mx{(u')})^+\mx{(u'v')} = (\mx{u})^+\mx{(uv)}$, where for the second equality we used a defining relation of $F$, so $F$ satisfies the identity \eqref{eq:left_s} and the statement follows.

(2) and (4) are proved similarly, and (3) follows from (1) and (2).
\end{proof}   

\begin{lemma}\label{lem:quotient1a}
Let $F$ be an $F$-birestriction monoid and let $\rho$ be a $(2,1,1,0)$-congruence on $F$ such that $\rho\subseteq \sigma$. Then $\rho$ is in fact a $(2,1,1,1,0)$-congruence. 
\end{lemma}

\begin{proof}
For $a\in F$ let $[a]_{\rho} = \{b\in F\colon a\mathrel{\rho} b\}\in F/\rho$ be the $\rho$-class of $a$ and $[a]_{\sigma} = \{b\in F\colon b\mathrel{\sigma} a\} \in F/\sigma$ be the $\sigma$-class of $a$. By assumption we have $[a]_{\rho} \subseteq [a]_{\sigma}$. For any $b \in [a]_{\sigma}$ we have $\mx{a}\geq b$, so that $b=\mx{a}b^*$, which implies $[b]_{\rho} = [\mx{a}]_{\rho}[b]_{\rho}^*$, so that $[\mx{a}]_{\rho}\geq [b]_{\rho}$. Since any element which is $\sigma$-related with $[\mx{a}]_{\rho}$ in $F/\rho$ has the form $[b]_{\rho}$ where $b \in [a]_{\sigma}$, it follows that  $[\mx{a}]_{\rho} = \mx{[a]_{\rho}}$, so that $\rho$ respects the operation $u \mapsto \mx{u}$ and is thus a $(2,1,1,1,0)$-congruence.
\end{proof}

\begin{lemma} \label{lem:cong12a} Let $F$ be an $F$-birestriction monoid and $R \subseteq F\times F$ a relation on $F$ such that $R\subseteq \sigma$. Then the $(2,1,1,0)$-congruence on $F$ generated by $R$ coincides with the $(2,1,1,1,0)$-congruence on $F$ generated by $R$.
\end{lemma}

\begin{proof}
Let $\rho_1$ be the $(2,1,1,1,0)$-congruence on $F$ generated by $R$ and $\rho_2$ the $(2,1,1,0)$-congruence on $F$ generated by $R$. It is clear that $\rho_2\subseteq \rho_1$. Since $\sigma$ is a $(2,1,1,1,0)$-congruence, we have $R\subseteq \rho_2\subseteq \rho_1 \subseteq \sigma$. By Lemma \ref{lem:quotient1a} it follows that $\rho_2$ is a $(2,1,1,1,0)$-congruence, so that $\rho_2=\rho_1$ by the minimality of $\rho_1$.
\end{proof}

We remind the reader that the set of relations $N$ is defined in \eqref{def:N} as 
$$
N = \{\overline{x}\geq x, \overline{uv} \geq \overline{u}\,\overline{v}\colon x\in X, u,v \in X^+\}.
$$
We also recall that the birestriction monoid ${\mathcal F}$ and the inverse monoid ${\mathcal I}$ are defined by
$${\mathcal F} = {\mathrm{BRestr}} \langle X\cup \overline{X^+} \mathrel{\vert} N\rangle, \,\,
{\mathcal I} = {\mathrm{Inv}}\langle X\cup \overline{X^+} \mathrel{\vert} \psi(N) \rangle,$$
where $\psi\colon {\mathsf{FBR}}(X\cup \overline{X^+}) \to {\mathsf{FI}}(X\cup \overline{X^+})$ is the $(X\cup \overline{X^+})$-canonical $(2,1,1,0)$-morphism.

We now define the following sets of relations on ${\mathsf{FBR}}(X\cup \overline{X^+})$:
\begin{equation}\label{def:Nls}
N_{ls} = \{\overline{x} \geq x, \, \overline{u}\,\overline{v} = \overline{u}^+\overline{uv} \colon x\in X,\, u,v\in X^+\},
\end{equation}
\begin{equation}\label{def:Nrs}
N_{rs} = \{\overline{x} \geq x, \, \overline{u}\,\overline{v} = \overline{uv}\,\overline{v}^* \colon x\in X, \,
u,v\in X^+\},
\end{equation}
\begin{equation}\label{def:Np1}
N_p = \{\overline{x} \geq x, \, \overline{uv} = \overline{u}\,\overline{v}\colon x\in X, \, u,v\in X^+\},
\end{equation}
$N_s = N_{ls}\cup N_{rs}$ and $N_{\varnothing} = N$.
We further put
$${\mathcal F}_{i} =  {\mathrm{BRestr}}\langle X\cup \overline{X^+} \mathrel{\vert} N_{i}\rangle, \,\, i\in \{ls, rs,s,p, \varnothing\}.$$

Note that ${\mathcal F}_{\varnothing} = {\mathcal F}$. Let $\rho_N$ be the $(2,1,1,0)$-congruence on ${\mathsf{FBR}}(X\cup \overline{X^+})$ generated by $N$ and $\rho_{N_i}$ the $(2,1,1,1,0)$-congruence on ${\mathsf{FBR}}(X\cup \overline{X^+})$ generated by $N_i$, for each $i\in\{ls, rs, s, p\}$. Since relations $N_i$ imply relations $N$, we have $\rho_{N}\subseteq \rho_{N_i}$ for each $i\in\{ls, rs, s, p\}$. From \cite[Theorem 6.15]{BS81} it follows that ${\mathcal F}_i$ is a $(2,1,1,0)$-quotient of ${\mathcal F}$ by the congruence $\rho_{N_i}/\rho_{N}$.
Since ${\mathcal F}/\sigma \simeq {\mathcal F_i}/\sigma\simeq X^*$, the congruence $\rho_{N_i}/\rho_{N}$ on ${\mathcal F}$  is contained in $\sigma$, so that it is a $(2,1,1,1,0)$-congruence by Lemma \ref{lem:quotient1a}. Since the image of the $(2,1,1,1,0)$-congruence $\rho_{N_i}/\rho_{N}$ under the isomorphism of ${\mathcal F}$ and ${\mathsf{FFBR}}(X)$ coincides with the $(2,1,1,1,0)$-congruence generated by the defining relations of ${\mathsf{FFBR}}_{i}(X)$ given in Proposition \ref{prop:presenations_f}(1), we have proved the following. 

\begin{theorem}\label{th:structure_rel_free}\mbox{} For each  $i\in \{ls, rs,s,p\}$ the  $F$-birestriction monoids ${\mathsf{FFBR}}_{i}(X)$ and ${\mathcal F}_{i}$ are $X$-canonically $(2,1,1,1,0)$-isomorphic and $(X\cup \overline{X^+})$-canonically $(2,1,1,0)$-isomorphic.
\end{theorem}

Towards the coordinatization of ${\mathsf{FFBR}}_i(X)$, we further put 
\begin{equation}\label{eq:a22c}
{\mathcal I}_{i}=\mathrm{Inv}\langle X\cup \overline{X^+} \mathrel{\vert}\psi(N_i)\rangle,
\end{equation}
for each $i \in \{ ls,rs,s,p,\varnothing\}$ 
to be the universal  inverse monoid of the $(X\cup \overline{X^+})$-generated birestriction monoid ${\mathcal F}_i$. Note that ${\mathcal I}_{\varnothing}= {\mathcal I}$. By Proposition \ref{prop:isom_sem_rel} we have that $P({\mathcal F}_i) \simeq E({\mathcal I}_i)$. Since ${\mathcal F}_i$ is $F$-birestriction, it is proper and thus decomposes into the partial action product
\begin{equation}\label{eq:m9b}
{\mathcal F}_i \simeq P({\mathcal F}_i) \rtimes 
{\mathcal F}_i/\sigma \simeq E({\mathcal I}_{i}) \rtimes 
X^*,  \,\, i\in \{ls, rs,s,p,\varnothing\}.
\end{equation}

We have proved the following analogue of Theorem \ref{th:main1}.

\begin{theorem}\label{th:main1aa}
Let $X$ be a non-empty set and $i\in \{ls,rs,s,p\}$. The map $x\mapsto (x^+, x)$ extends to an  $X$-canonical $(2,1,1,1,0)$-isomorphism
$$
{\mathsf{FFBR}}_i(X) \simeq E({\mathcal I}_i)\rtimes X^*.
$$    
\end{theorem}

\section{The word problem for the free (left, right) strong  
$F$-birestriction monoids is decidable}\label{sec:word_rel}
By Theorem \ref{th:structure_rel_free} and \eqref{eq:m9b}, the word problem for each of ${\mathsf{FFBR}_{i}}(X)$, $i \in \{ls,rs,s,p\}$, reduces to the word problem for the inverse monoid ${\mathcal I}_{i}$.

\begin{proposition} 
For each $i \in \{ls,rs,s\}$ the Sch\"utzenberger graphs of elements of ${\mathcal I}_{i}$ are finite and effectively constructed. Consequently, the word problem for ${\mathcal I}_{i}$ is decidable.
\end{proposition}

\begin{proof}
We argue similarly as in the proof of Theorem \ref{th:word} and start from the determinized form of the linear graph of  $u\in (Y\cup Y^{-1})^*$ where $Y=X\cup \overline{X^+}$. 
If it is closed, we are done. Otherwise, we apply to it a series of elementary $P$-expansions and determinizations. These include the ones considered in Cases 1 and 2 of the proof of Theorem \ref{th:word} and possibly other elementary $P$-expansions which we now describe.

Let us consider the case of ${\mathcal{I}_{ls}}$. It remains only to treat the relations $\overline{u}\,\overline{v} = \overline{u}\,\overline{u}^{-1}\,\overline{uv}$, where $u,v\in X^+$. 
The only possibility when such a relation causes attaching a path is when there are edges  $(\alpha,\overline{u},\beta)$ and $(\alpha,\overline{uv},\gamma)$ in the graph, but there is no edge $(\beta,\overline{v},\gamma)$. Then after applying an elementary $P$-expansion (attaching a path from $\alpha$ to $\gamma$ labeled by $\overline{u}\,\overline{v}$) and several determinizations, we arrive at the graph obtained from the initial graph by adding the edge $(\beta, \overline{v}, \gamma)$, see Figure \ref{fig:cls}. If applicable, we determinize this graph and obtain a determinized graph with at most the same number of vertices as the initial graph and with fewer number of pairs of edges of the form  $(\alpha,\overline{u},\beta)$ and $(\alpha,\overline{uv},\gamma)$ such that the there is no edge $(\beta,\overline{v},\gamma)$ in the graph. It follows that in a finite number of steps we obtain a closed graph 
with at most the same number of vertices as the initial graph.  
This completes the proof in this case.

The case of ${\mathcal{I}}_{rs}$ is symmetric and that of ${\mathcal{I}}_{s}$ is a combination of the cases of ${\mathcal{I}}_{ls}$ and ${\mathcal{I}}_{rs}$.

\begin{figure}[!ht]
\centering
\begin{tikzpicture}[scale=0.6]
\begin{scope}[every node/.style={circle,fill,inner sep=1.5pt}]
\node[label={[label distance=0.1cm]:$\alpha$}] (A) at (0,0) {}; 
\node[label={[label distance=0.1cm]:{$\beta$}}] (B) at (2.5,0) {};
\node[label={[label distance=0.1cm]:{$\gamma$}}] (C) at (5,0) {};
\end{scope}
\begin{scope}[>={Stealth[black]}, every node/.style={fill=white,circle,scale=0.7}, every edge/.style={draw, thick}]
\path [->] (A) edge[dashed,bend left=25] node[above=0.5pt] {$\overline{u}$} (B);
\path [->] (A) edge[dashed, bend right=25] node[below=0.5pt] {$\overline{uv}$} (C);
\end{scope}

\node (P) at (6, 0) {};
\node (Q) at (8.5, 0) {};

\begin{scope}, 
\draw[-stealth,thick,decorate,decoration={snake,amplitude=.3mm},label=abc] (P) -- (Q) node[above,midway]{};
\end{scope}

\begin{scope}[every node/.style={circle,fill,inner sep=1.5pt},xshift=9.5cm]
\node[label={[label distance=0.1cm]:$\alpha$}] (A) at (0,0) {}; 
\node[label={[label distance=0.1cm]:{$\beta$}}] (B) at (2.5,0) {};
\node[label={[label distance=0.1cm]:{$\gamma$}}] (C) at (5,0) {};
\node[label={[label distance=0.1cm]:{}}] (D) at (2.5,1.5) {};
\end{scope}

\begin{scope}[>={Stealth[black]}, every node/.style={fill=white,circle,scale=0.7}, every edge/.style={draw, thick}]
\path [->] (A) edge[dashed,bend left=25] node[above=0.5pt] {$\overline{u}$} (B);
\path [->] (A) edge[dashed,bend left=25] node[above=0.5pt] {$\overline{u}$} (D);
\path [->] (D) edge[dashed, bend left=25] node[above=0.5pt] {$\overline{v}$} (C);
\path [->] (A) edge[dashed, bend right=25] node[below=0.5pt] {$\overline{uv}$} (C);
\end{scope}

\node (P) at (6, -4) {};
\node (Q) at (8.5, -4) {};
\begin{scope}, 
\draw[-stealth,thick,decorate,decoration={snake,amplitude=.3mm},label=abc] (P) -- (Q) node[above,midway]{};
\end{scope}
\begin{scope}[every node/.style={circle,fill,inner sep=1.5pt},xshift=9.5cm]
\node[label={[label distance=0.1cm]:$\alpha$}] (A) at (0,-4) {}; 
\node[label={[label distance=0.1cm]:{$\beta$}}] (B) at (2.5,-4) {};
\node[label={[label distance=0.1cm]:{$\gamma$}}] (C) at (5,-4) {};
\end{scope}

\begin{scope}[>={Stealth[black]}, every node/.style={fill=white,circle,scale=0.7}, every edge/.style={draw, thick}]
\path [->] (A) edge[dashed,bend left=25] node[above=0.5pt] {$\overline{u}$} (B);
\path [->] (B) edge[dashed, bend left=25] node[above=0.5pt] {$\overline{v}$} (C);
\path [->] (A) edge[dashed, bend right=25] node[below=0.5pt] {$\overline{uv}$} (C);
\end{scope}
\end{tikzpicture}
\caption{Illustration of a $P$-expansion: the case of ${\mathcal{I}_{ls}}$.}\label{fig:cls}
\end{figure}
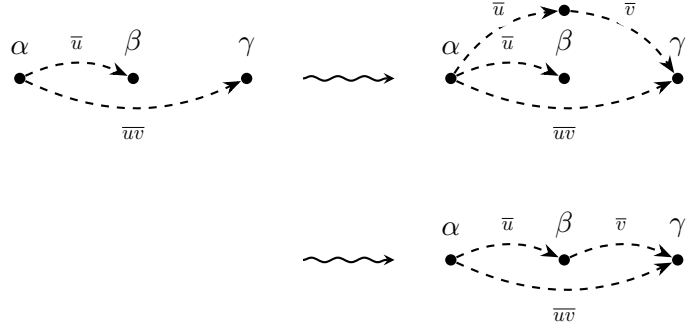
\end{proof}
We have proved following.

\begin{theorem}\label{thm:word_rel}
The word problem for ${\mathsf{FFBR}_{ls}}(X)$, ${\mathsf{FFBR}_{rs}}(X)$ and ${\mathsf{FFBR}_s}(X)$ is decidable.   \end{theorem}

The word problem for ${\mathsf{FFBR}}_p(X)$ could be treated quite similarly, though additional attention is required since an elementary $P$-expansion based on the relation $\overline{uv} = \overline{u}\,\overline{v}$ may cause adding new vertices to the approximate graphs of the Sch\"utzenberger graphs. We opted to postpone treating the word problem in ${\mathsf{FFBR}}_p(X)$ till Subsection \ref{subsec:perf}, where we provide a transparent geometric model for ${\mathsf{FFBR}}_p(X)$ based on the Cayley graph of ${\mathsf{FG}}(X)$ with respect to generators $X\cup \overline{X}$ which yields the decidability of the word problem.

\section{Geometric models of $F$-birestriction monoids}\label{sec:geom_models}
In what follows by $({\mathsf{FG}}(X), X\cup \overline{X})$ we denote the group ${\mathsf{FG}}(X)$ considered as an $(X\cup {\overline{X}})$-generated group via the assignment map $x\mapsto x$ and $\overline{x}\mapsto x$ where $x\in X$. Furthermore, by $({\mathsf{FG}}(X), X\cup \overline{X^+})$ we denote the group ${\mathsf{FG}}(X)$ considered as an $(X\cup {\overline{X^+}})$-generated group via the assignment map $x\mapsto x$ and $\overline{u}\mapsto u$ where $x\in X$ and $\overline{u}\in \overline{X^+}$.

In this section we explore if 
the free $X$-generated $F$-birestriction monoids can be embedded into quotients of the Margolis-Meakin expansions  $M({\mathsf{FG}}(X), X\cup \overline{X^+})$ and $M({\mathsf{FG}}(X), X\cup \overline{X})$ constructed via dual-closure operators described in Subsection~\ref{subs:e_unitary}.

\subsection{The cases of ${\mathsf{FFBR}}(X)$, ${\mathsf{FFBR}_{ls}}(X)$, ${\mathsf{FFBR}_{rs}}(X)$
and ${\mathsf{FFBR}_s}(X)$} 
We first treat the case where $|X|\geq 2$ and then that where $|X|=1$.
\begin{theorem}\label{th:nogeom}
Let $|X|\geq 2$ and $i\in \{\varnothing,ls,rs,s\}$. Then
${\mathcal F}_i$ does not admit a geometric model.
\end{theorem}

\begin{proof} 
We show that ${\mathcal I}_i$ is not $E$-unitary for each $i\in \{\varnothing,ls,rs,s\}$. For distinct elements $x,y\in X$ we put $a=\overline{xy}(\overline{y^2})^{-1}\overline{yx}(\overline{x^2})^{-1}$ and observe that $[a]_{{\mathsf{FG}}(X)} = xyy^{-2}yxx^{-2} = 1$. If ${\mathcal I}_i$ were $E$-unitary, \eqref{eq:i2} would yield that the value of $a$ in $M(G,X\cup \overline{X^+})$ and thus also in its quotient ${\mathcal I}_i$ would be an idempotent. However, drawing the linear graph of $a$ it is easy to see that it does not admit determinizations or elementary $P$-expansions and is thus the Sch\"utzenberger graph of $a$, see Figure \ref{fig:c1}. 
For example, in the case of ${\mathcal I}_{\varnothing}$ no elementary $P$-expansions in the linear graph of $a$ are possible based on its defining relations $\psi(N_{\varnothing})$.
Since the graph does not accept $a^2$, \cite[Corollary 3.2]{S90} implies that $a$ is not an idempotent in ${\mathcal I}_i$, so that the latter is not $E$-unitary. It follows from Proposition \ref{prop:model} that ${\mathcal F}_i$ does not admit a geometric model.

\begin{figure}[!ht]
\begin{tikzpicture}[scale=0.6]
\begin{scope}[every node/.style={circle,fill,inner sep=1.5pt}]
\node[label= {[label distance=0.3cm]:}] (A) at (0,0) {};
\node[label= {[label distance=0.1cm]:}] (B) at (2.5,0) {};
\node[label= {[label distance=0.03cm]:}] (C) at (5,0) {};
\node[label= {[label distance=0.03cm]:}] (D) at (7.5,0) {};
\node[label= {[label distance=0.03cm]:}] (E) at (10,0) {};
\end{scope}

\begin{scope}[>={Stealth[black]}, every node/.style={fill=white,circle,scale=0.7}, every edge/.style={draw, thick}]
\path [->] (0,-0.83) edge[] (A);
\path [->] (A) edge[dashed,bend left=25] node[above=0.5pt] {$\overline{xy}$} (B);
\path [<-] (B) edge[dashed, bend left=25] node[above=0.5pt] {$\overline{y^2}$} (C);
\path [->] (C) edge[dashed,bend left=25] node[above=0.5pt] {$\overline{yx}$} (D);
\path [<-] (D) edge[dashed, bend left=25] node[above=0.5pt] {$\overline{x^2}$} (E);
\path [->] (E) edge[] (10,-0.83);
\end{scope}
 
\end{tikzpicture}
\caption{The Sch\"utzenberger graph of $a$.}\label{fig:c1}
\end{figure}
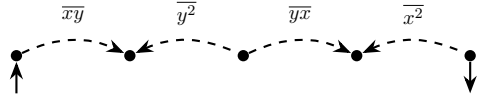
\end{proof}

\begin{proposition}\label{prop:nogeom1}
Let $X=\{x\}$. Then ${\mathcal F}_{\varnothing}$ does not admit a geometric model. 
\end{proposition}

\begin{proof} Similarly as in the proof of Proposition \ref{th:nogeom}, we see that the element $\overline{x}\overline{x^2}^{-1}\overline{x}$  has value $1$ in ${\mathsf{FG}}(X)$, but is not an idempotent in ${\mathcal I}_{\varnothing}$ since the Sch\"utzenberger graph of $\overline{x}\overline{x^2}^{-1}\overline{x}$ presented on Figure \ref{fig:o1} does not accept $(\overline{x}\overline{x^2}^{-1}\overline{x})^2$.  The statement follows applying Proposition \ref{prop:model}.
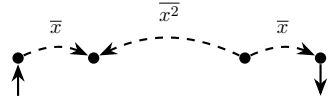
\begin{figure}[!ht]
\begin{tikzpicture}[scale=1]
\begin{scope}[every node/.style={circle,fill,inner sep=1.5pt}]
\node[label= {[label distance=0.1cm]:}] (A) at (0,0) {};
\node[label= {[label distance=0.1cm]:{}}] (B) at (1,0) {};
\node[label= {[label distance=0.1cm]:{}}] (C) at (3,0) {};
\node[label= {[label distance=0.1cm]:{}}] (D) at (4,0) {};
\end{scope}
\begin{scope}[>={Stealth[black]}, every node/.style={fill=white,circle,scale=0.7}, every edge/.style={draw, thick}]
\path [->] (0,-0.5) edge[] (A);
\path [->] (A) edge[dashed,bend left=25] node[above=0.5pt] {$\overline{x}$} (B);
\path [<-] (B) edge[dashed, bend left=25] node[above=0.5pt] {$\overline{x^2}$} (C);
\path [->] (C) edge[dashed, bend left=25] node[above=0.5pt] {$\overline{x}$} (D);
\path [->] (D) edge[] (4,-0.5);
\end{scope}
\end{tikzpicture}
\caption{The Sch\"utzenberger graph of $\overline{x}\overline{x^2}^{-1}\overline{x}$.}\label{fig:o1} 
\end{figure}
\end{proof}

In contrast to the previous statement, we have the following.

\begin{theorem}\label{th:geom1}
Let $X=\{x\}$ and $i\in \{ls,rs,s\}$. Then ${\mathcal F}_i$ admits a geometric model. 
\end{theorem}

\begin{proof} 
Observe that the morphism  $\psi\colon {\mathcal F}_i/\sigma \to {\mathcal I}_i/\sigma$ is injective since it is the embedding $X^*\to {\mathsf{FG}}(X)$. Since in addition ${\mathcal F}_i$ is proper, Proposition \ref{prop:model}(2) implies that  it suffices to show that the inverse monoid  ${\mathcal I}_i$ is $E$-unitary.
Putting $Y = X\cup \overline{X^+} = \{x\} \cup \{\overline{x^k}\colon k\in {\mathbb N}\}$, let $u\in (Y\cup Y^{-1})^*$ be such that $[u]_{{\mathsf{FG}}(X)}=1$. We show that $[u]_{\mathcal{I}_i}$ is an idempotent. 
From now till the end of this proof we sometimes use the same notation for words in $(Y\cup Y^{-1})^*$ and their values in ${\mathcal I}_i$. It will be always clear from the context which objects we work with. For example, we will denote $[u]_{ {\mathcal I}_i}$ simply by $u$.

Proposition \ref{prop:generation} implies that there exist $v,e\in (Y\cup Y^{-1})^*$ such that $u=ve$ and $[e]_{{\mathcal I}_i}$ is an idempotent. We show that $v$ can be chosen so that it does not contain the letters $x$ and $x^{-1}$. Suppose $v$ contains $x$ and write $v=w_1xw_2$. Since in ${\mathcal I}_i$ we have  $\overline{x}\geq x$, it follows that $x=\overline{x}x^*$ so that $u=w_1\overline{x}x^*w_2e = w_1\overline{x}w_2fe$ where $f=(x^*w_2)^* \in E({\mathcal I}_i)$. If $v = w_1x^{-1}w_2$ then $v^{-1} = w_2^{-1}xw_1^{-1}$ and we can similarly rewrite $u^{-1} = e(w_2^{-1}x^+)^+w_2^{-1}\overline{x}w_1^{-1}$ so that $u=w_1\overline{x}^{-1}w_2fe$ with $f=(w_2^{-1}x^+)^+$. Arguing by induction, the claim follows.

Therefore, we have $u = y_1\cdots y_ne$ where $n\geq 0$, all $y_i$ are from $\{\overline{x^k}\colon k\in {\mathbb N}\}$ or $\{\overline{x^k}^{\,-1}\colon k\in {\mathbb N}\}$ and $[e]_{{\mathcal I}_i}$ is an idempotent. Since $[u]_{{\mathsf{FG}}(X)} = 1$ and $[e]_{{\mathsf{FG}}(X)} = 1$, we have that $[y_1\cdots y_n]_{{\mathsf{FG}}(X)}=1$ as well. It suffices to show that $y_1\cdots y_n$ is an idempotent in ${\mathcal I}_i$ (then also $u$ is an idempotent as a product of idempotents). We apply induction on $n$. If $n=0$, $y_1\cdots y_n=1$, and we are done. Let $n\geq 1$. We show that $u$ can be rewritten as a product $u=z_1\cdots z_m g$ where $0\leq m\leq n-1$ and $[g]_{{\mathcal I}_i}$ is an idempotent.

{\em Case 1.} Suppose that $y_1\cdots y_n$ has  a subword $y_iy_{i+1} = \overline{x^k}\,\overline{x^m}$.
Since $\overline{x^k}\,\overline{x^m} = \overline{x^{k+m}}f$ where $f=(\overline{x^k}\overline{x^m})^*$, applying  \eqref{eq:ample} we can write $u=y_1\cdots y_{i-1}\overline{x^{k+m}}y_{i+2}\cdots y_n g$ where $[g]_{{\mathcal I}_i}$ is an idempotent.  

{\em Case 2.} Suppose that $y_1\cdots y_n$ has  a subword $y_iy_{i+1} = \overline{x^k}^{-1}\,\overline{x^m}^{-1}$. Then $u^{-1}$ contains a subword $\overline{x^m}\,\overline{x^k}$, so we can write $u^{-1} = z_1\cdots z_{n-1}g$ where $[g]_{{\mathcal I}_i}$ is an idempotent. Then $u=g^{-1}z_{n-1}^{-1}\cdots z_1^{-1} = z_{n-1}^{-1}\cdots z_1^{-1}h$ where $h=(g^{-1}z_{n-1}^{-1}\cdots z_1^{-1})^*$. The claim follows.

{\em Case 3.} Suppose  that $y_1\cdots y_n$ has a subword $y_iy_{i+1}$ where $y_{i+1} = y_i^{-1}$. Then $y_iy_{i+1}$ is an idempotent, and using \eqref{eq:ample} we can write $u=y_1\cdots y_{i-1}y_{i+2}\cdots y_n g$ where $[g]_{{\mathcal I}_i}$ is an idempotent, as needed.

{\em Case 4.} Suppose  that  the letters from $\{\overline{x^k}\colon k\in {\mathbb N}\}$ and $\{\overline{x^k}^{\,-1}\colon k\in {\mathbb N}\}$ in $y_1\cdots y_n$ alternate and $y_{i+1}\neq y_i^{-1}$ for all $i$. From now on we work with the case of ${\mathcal F}_{ls}$ so that ${\mathcal I}_{ls}$ is left strong and satisfies \begin{equation}\label{eq:strong1}
\overline{x^k}\,\overline{x^{l-k}} = \overline{x^k}^+ \overline{x^{l}},
\end{equation}
for all $l>k$.
We show that we necessarily have $n\geq 3$. Indeed, if $n=1$ then $y_1\cdots y_n = y_1 = \overline{x^k}$ for $k\in {\mathbb N}$ which can not have value $1$ in ${\mathsf{FG}}(X)$ so this case is impossible. If $n=2$, the only possibility for $y_1y_2$ to have value $1$ in ${\mathsf{FG}}(X)$ is that $y_2 = y_1^{-1}$, which is excluded. It follows that $n\geq 3$ and one of the products $y_1y_2$ or $y_2y_3$ is of the form $\overline{x^k}^{-1}\overline{x^l}$ where $k\neq l$. We consider two subcases.

{\em Subcase 1.} Suppose that $k<l$. Then $$\overline{x^k}^{-1}\overline{x^l} = \overline{x^k}^{-1}\overline{x^k}\,\overline{x^k}^{-1}\overline{x^l} = \overline{x^k}^{-1}\overline{x^k}^+\overline{x^l} = \overline{x^k}^{-1}\overline{x^k}\,\overline{x^{l-k}} = \overline{x^k}^* \overline{x^{l-k}},
$$
where for the third equality we used \eqref{eq:strong1}.
Applying \eqref{eq:ample}, we can move $ \overline{x^k}^*$ to the right-hand side of the product. The claim follows.

{\em Subcase 2.} Suppose that $k > l$. 
Then $u^{-1}$ contains the subword $\overline{x^l}^{\,-1}\overline{x^k}$ and by the previous subcase we can write $u^{-1} = z_1\cdots z_m g$ where $0\leq m\leq n-1$ and $g$ is an idempotent. But then $u = g z_m^{-1}\cdots z_1^{-1}$ and applying \eqref{eq:ample} this equals $z_m^{-1}\cdots z_1^{-1}(g z_m^{-1}\cdots z_1^{-1})^*$, as needed.

The case of ${\mathcal F}_{rs}$ follows by symmetry, and that of ${\mathcal F}_{s}$ by combining the cases of ${\mathcal F}_{ls}$ and ${\mathcal F}_{rs}$.
\end{proof}

\begin{remark} One can show that the inverse monoids ${\mathcal I}_{ls}$, ${\mathcal I}_{rs}$ and ${\mathcal I}_{s}$ are in fact $F$-inverse. The maximum element of the $\sigma$-class of projections is $1$, and the maximum element of the $\sigma$-class which projects down to $g=x^k$, where $k\in {\mathbb Z}\setminus \{0\}$, is $\overline{x^k}$, if $k>0$, and $\overline{x^{-k}}^{\,-1}$, if $k<0$. Since we do not use this fact in the sequel, we omit the proof.
\end{remark}

\begin{remark}\label{rem:isom1a}
Let $X=\{x\}$ and $i\in \{ls,rs,s\}$. Let $\varphi$ and $\tilde{\varphi}$ be the underlying premorphisms of $\psi({\mathcal F}_{i})$ and ${\mathcal I}_{i}$, respectively. By \eqref{eq:domains} for every $t\in \psi({\mathcal F}_{i})$ we have ${\mathrm{dom}}(\varphi_t)\subseteq {\mathrm{dom}}(\tilde{\varphi}_t)$. Let us show that the reverse inclusion also holds. Suppose that $e\in {\mathrm{dom}}{\tilde\varphi}_t$ where $t\in X^*$. If $t=1$, $e\in E({\mathcal I}_i) = P(\psi({\mathcal F}_i))$ so that $e\in {\mathrm{dom}}\varphi_1$. If $t\in X^+$ it is easy to see that the $\sigma$-class of ${\mathcal I}_i$ which projects onto $t$ has a maximum element which is the element $\overline{t}$. It follows that $e\in {\mathrm{dom}}(\tilde\varphi_t)$ if and only if $e\leq \overline{t}^*$ in which case also $e\in {\mathrm{dom}}(\varphi_t)$. It follows from Lemma \ref{lem:varphi_tilde} that $\psi({\mathcal F}_i)$ is isomorphic to $E({\mathcal I}_i) \rtimes_{\tilde\varphi} X^*$.
\end{remark}

Let $i\in \{ls, rs, s,\varnothing\}$ and put ${\mathsf{FFBR}}_{\varnothing}(X) = {\mathsf{FFBR}}(X)$.
The results of this subsection imply that  ${\mathsf{FFBR}}_i(X)$ can be $(X\cup \overline{X^+})$-canonically $(2,1,1,0)$-embedded into a quotient of $M({\mathsf{FG}}(X), X\cup \overline{X^+})$ constructed via dual-closure operators described in Subsection \ref{subs:e_unitary} if and only if $|X|=1$ and $i\in \{ls,rs,s\}$.

\subsection{${\mathsf{FFBR}}_p(X)$ admits a geometric model and has decidable word problem}\label{subsec:perf}

We first observe that, due to the relation $\overline{uv} = \overline{u} \,\overline{v}$, $u,v\in X^+$, ${\mathcal F}_p$ is in fact $(X\cup \overline{X})$-generated. Put $$F_p={\mathrm{BRestr}}\langle X\cup \overline{X} \mathrel{\vert} \overline{x} \geq x, x\in X \rangle.$$
For $u=x_1\cdots x_n\in X^+$, where $x_i\in X$ for all $i$, put $\overline{u} = \overline{x_1}\,\cdots \,\overline{x_n}\in F_p$. Then $F_p$ is $(X\cup \overline{X^+})$-generated and $\overline{uv} = \overline{u}\,\overline{v}$ holds in $F_p$ for all $u,v\in X^+$.
We obtain the following.

\begin{lemma} \label{lem:a23b} ${\mathcal F}_p$ and ${\mathsf{FFBR}}_p(X)$ are $(2,1,1,0)$-generated by $(X\cup \overline{X})$. They are $(X\cup \overline{X})$-canonically and $(X\cup \overline{X^+})$-canonically $(2,1,1,0)$-isomorphic to $F_p$.
\end{lemma}

It follows that ${\mathcal I}_p$ is $(X\cup \overline{X})$-generated and is $(X\cup \overline{X^+})$-canonically isomorphic to
\begin{equation}\label{eq:a22k}
I_p = {\mathrm{Inv}}\langle X\cup \overline{X} \mathrel{\vert} \overline{x} \geq x, x\in X \rangle.
\end{equation}

\begin{proposition}\label{prop:e_unitary}
The inverse monoid $I_p$ is $E$-unitary. 
\end{proposition}
For the proof we will need the following lemma.

\begin{lemma} \label{lem:a23a} Let $u\in (X\cup X^{-1})^*$ be such that $[u]_{{\mathsf{FG}}(X)} =1$. Then $u$ is either empty or can be written as a product
$$
u=x_1 v_1 x_1^{-1} x_2 v_2 x_2^{-1} \cdots x_n v_n x_n^{-1},
$$
where $n\geq 1$ and for all $i\in \{1,\dots, n\}$ we have that $x_i\in X\cup X^{-1}$ and $v_i\in (X\cup X^{-1})^*$ with $[v_i]_{{\mathsf{FG}}(X)}=1$.
\end{lemma}

\begin{proof}
Suppose that $u\neq 1$. By assumption, there is a sequence $u=u_0, u_1,\dots, u_k = 1$ where $k\geq 1$ and $u_i\in (X\cup X^{-1})^*$, $1\leq i\leq k$, are such that every $u_{i+1}$ is obtained from $u_i$ by removing from it a factor of the form $xx^{-1}$, where $x\in X\cup X^{-1}$. Let $x\in X\cup X^{-1}$ be the first letter of $u$. Then $u$ has the letter $x^{-1}$, at some position $l$, and there is  $i\in \{0,\dots, k-1\}$ such that the first letter $x$ of $u$ and the $l$-th letter $x^{-1}$ of $u$ were not removed during the passage from $u$ to $u_i$, appear in $u_i$ as the factor $xx^{-1}$ which is removed at the passage to $u_{i+1}$. This means that we can write $u=xv_1x^{-1}u'$ where $v_1\in (X\cup X^{-1})^*$ and $[v_1]_{{\mathsf{FG}}(X)}=1$. But then $[xv_1x^{-1}]_{{\mathsf{FG}}(X)}=1$, so that $[u']_{{\mathsf{FG}}(X)} = 1$ as well. The statement now follows applying induction.
\end{proof}
   
\begin{proof}[Proof of Proposition \ref{prop:e_unitary}.] Let $Y=X\cup \overline{X}$. We consider $G_p=I_p/\sigma$ as a $Y$-generated group that is canonically isomorphic to the $Y$-generated group $({\mathsf{FG}}(X),Y)$.
From \eqref{eq:a22k} and Lemma \ref{lem:generation2a}(2)
we have that
$G_p= {\mathrm{Gr}}\langle X\cup \overline{X} \mathrel{\vert} \overline{x} = x, x\in X \rangle.$

Let $u\in (Y\cup Y^{-1})^*$ be such that $[u]_{G_p} = 1$. We show that $[u]_{I_p}$ is an idempotent. If $u=1$, this is clear, so we assume that $|u|\geq 1$. The map $Y\to Y$, given by $x, \overline{x}\mapsto x$, extends to a morphism  $f\colon (Y\cup Y^{-1})^* \to (X\cup X^{-1})^*$ and we have $[u]_{G_p} = [f(u)]_{G_p}$. Hence $f(u) \in (X\cup X^{-1})^*$ and $[f(u)]_{{\mathsf{FG}}(X)}=1$. In view of Lemma \ref{lem:a23a}, we can write
$$
u=x_1 v_1 z_1^{-1} x_2 v_2 z_2^{-1} \cdots x_n v_n z_n^{-1},
$$
for some $n\geq 1$ where for all $i\in \{1,\dots, n\}$ we have  that $x_i,z_i\in Y\cup Y^{-1}$ are such that $f(x_i)=f(z_i)$ and $v_i\in (Y\cup Y^{-1})^*$ satisfy $[v_i]_{G_p}=1$. Arguing by induction on the length of $u$, we can assume that $[v_i]_{I_p}$ are idempotents for all $i$. Since idempotents are closed with respect to the multiplication, it suffices to prove that each $[x_iv_iz_i^{-1}]_{I_p}$ is an idempotent. 

If $x_i = z_i$ this follows from the fact that in an inverse semigroup a conjugate $aea^{-1}$ of an idempotent $e$ is itself an idempotent. 

If $x_i = x\in X$ and $z_i = \overline{x}$, then
$[x_iv_iz_i^{-1}]_{I_p} = [xv_i \overline{x}^{-1}]_{I_p} = [x^+ \overline{x}v_i \overline{x}^{-1}]_{I_p} = [x^+]_{I_p}[\overline{x}v_i \overline{x}^{-1}]_{I_p}$, which is an idempotent. 

If $x_i=x^{-1}\in X^{-1}$ and $z_i = \overline{x}^{-1}$, then 
$[x_iv_iz_i^{-1}]_{I_p} = [x^{-1}v_i \overline{x}]_{I_p} = [x^* \overline{x}^{-1}v_i \overline{x}]_{I_p} = [x^*]_{I_p}[\overline{x}^{-1}v_i \overline{x}]_{I_p}$, which is also an idempotent.

The remaining two cases follow by symmetry.
\end{proof}

\begin{remark}
One can prove that in fact $I_p$ is $F$-inverse. If  $g = a_1^{\varepsilon_1}\cdots a_n^{\varepsilon_n}$, $a_i\in X$, $\varepsilon_i\in \{1,-1\}$, is a reduced word then the maximum element of the $\sigma$-class which projects down to $[g]_{{\mathsf{FG}}(X)}$ is $[\overline{a_1}^{\,\varepsilon_1}\cdots \overline{a_n}^{\,\varepsilon_n}]_{I_p}$. We remark, however, that $I_p$ is not perfect $F$-inverse, since, for example, $\overline{a}\,\overline{a}^{-1} \neq 1$ for any $a\in X$ (this follows from the description of the model of $I_p$, which is given below).
\end{remark}

\begin{theorem}\label{th:geom_p}
$F_p$ and ${\mathcal F}_p$ admit  geometric models.
\end{theorem}
    
\begin{proof}
By Proposition \ref{prop:e_unitary} we have that $I_p$ is $E$-unitary. Moreover,
$F_p/\sigma$ is isomorphic to the free monoid $X^*$ and it easily follows that $I_p/\sigma$
is isomorphic to the free group ${\mathsf{FG}}(X)$. Hence
the canonical morphism $\psi\colon F_p/\sigma \to I_p/\sigma$ is injective. Since, in addition $F_p$ is proper, Proposition \ref{prop:model} implies that $F_p$ admits a geometric model. For ${\mathcal F}_p$ the arguments are similar.
\end{proof}

Theorem \ref{th:geom_p} implies that ${\mathsf{FFBR}}_p(X)$ can be $(X\cup \overline{X})$-canonically $(2,1,1,0)$-embedded into a quotient of $M({\mathsf{FG}}(X), X\cup \overline{X})$ and $(X\cup \overline{X^+})$-canonically $(2,1,1,0)$-embedded into a quotient of $M({\mathsf{FG}}(X), X\cup \overline{X^+})$ constructed via dual-closure operators described in Subsection \ref{subs:e_unitary}.

We now describe the geometric model for $F_p$ based on $\Cay({\mathsf{FG}}(X), X\cup \overline{X})$. Recall that the assignment map {$X\cup \overline{X} \to {\mathsf{FG}}(X)$} is given by $x,\overline{x} \mapsto x$ for all $x\in X$. Similarly as in Remark \ref{rem:isom1a} one can show that $\psi(F_p)$ is isomorphic to $ E(I_p) \rtimes_{\bar\varphi} X^*$ where $\bar{\varphi}$ is the underlying premorphism of $I_p$.
The results of Subsection \ref{subs:e_unitary} imply that
$E(I_p)\simeq \bar{\rho}({\mathcal{X}}_{X\cup \overline{X}})$,
where ${\mathcal{X}}_{X\cup \overline{X}}$ is the semilattice of finite and connected subgraphs  of the Cayley graph $\Cay({\mathsf{FG}}(X), X\cup \overline{X})$ which contain the origin and $\rho$ is the congruence on $M({\mathsf{FG}}(X), X\cup \overline{X})$, such that $M({\mathsf{FG}}(X), X\cup \overline{X})/\rho \simeq I_p$. It follows that $F_p$ and thus also ${\mathsf{FFBR}}_p(X)$ is  $(X\cup \overline{X})$-canonically $(2,1,1,0)$-isomorphic and $X$-canonically $(2,1,1,1,0)$-isomorphic to the partial action product
\begin{equation}\label{eq:l18}
\bar{\rho}({\mathcal{X}}_{X\cup \overline{X}}) \rtimes X^*,
\end{equation}
where the partial action of $X^*$ on $\bar{\rho}({\mathcal{X}}_{X\cup \overline{X}})$ is by left translations (defined in Subsection \ref{subs:e_unitary}).
To describe the semilattice $\bar{\rho}({\mathcal{X}}_{X\cup \overline{X}})$, we  observe that the vertices of  $\Cay({\mathsf{FG}}(X), X\cup \overline{X})$ 
coincide with those of $\Cay({\mathsf{FG}}(X), X)$, and to each edge $(\alpha, x,\beta)$ of $\Cay({\mathsf{FG}}(X), X)$ correspond two  edges $(\alpha, x,\beta)$ and $(\alpha,\overline{x}, \beta)$ of $\Cay({\mathsf{FG}}(X), X\cup \overline{X})$. The congruence $\rho$ on $M({\mathsf{FG}}(X), X\cup \overline{X})$ is generated by the defining relations $\overline{x}\geq x$, $x\in X$, of $I_p$. A graph $\Gamma \in {\mathcal X}^c_{X\cup \overline{X}}$ is thus $\rho$-closed if it satisfies the condition
\begin{equation}\label{eq:condition}
\text{ if } (\alpha,x,\beta) \,\text{ is an edge of } \, \Gamma, \, \text{ then so is } \, (\alpha,\overline{x},\beta).
\end{equation}
It follows that the closure $\bar{\rho}(\Gamma)$ of $\Gamma\in {\mathcal X}^c_{X\cup \overline{X}}$ is the graph obtained from $\Gamma$ by adding to it all the missing `twin copies' $(\alpha,\overline{x},\beta)$ of all its edges $(\alpha,x,\beta)$ labeled by $X$ (and their inverse negative edges). (Remark, however, that a $\rho$-closed graph may contain an edge $(\alpha, \overline{x},\beta)$ but no edge $(\alpha, x,\beta)$.) In particular, $\bar{\rho}(\Gamma)$ has the same vertices as $\Gamma$ and if $\Gamma$ is finite then so is $\bar{\rho}(\Gamma)$. It follows that $\bar{\rho}({\mathcal{X}}_{X\cup \overline{X}})$ is the semilattice of all finite and connected subgraphs of $\Cay({\mathsf{FG}}(X), X\cup \overline{X})$ which contain the origin and satisfy condition \eqref{eq:condition}.

\begin{corollary}\label{cor:word_p}
The word problem for ${\mathsf{FFBR}_{p}}(X)$ is decidable.
\end{corollary}

\begin{proof}
We use the $X$-canonical $(2,1,1,1,0)$-isomorphism between ${\mathsf{FFBR}_{p}}(X)$ and the partial action product of \eqref{eq:l18}.
Two elements $(\bar{\rho}(\Gamma_1), u), (\bar{\rho}(\Gamma_2), v) \in \bar{\rho}({\mathcal{X}}_{X\cup \overline{X}}) \rtimes X^*$ 
are equal if and only if $u=v$ and $\bar{\rho}(\Gamma_1)= \bar{\rho}(\Gamma_2)$.
Since, for $\Gamma\in {\mathcal X}_{X\cup \overline{X}}$, $\bar{\rho}(\Gamma)$ is obtained from $\Gamma$ by adding to it finitely many edges (and no vertices), it is decidable if $\bar{\rho}(\Gamma_1)= \bar{\rho}(\Gamma_2)$.
\end{proof}

\section*{Acknowledgments}
We thank the referee for the comments and suggestions, which helped us improve the presentation.

\section*{Statements and declarations}

The authors have no competing interests to declare that are relevant to the content of this article.

\end{document}